\documentclass[12pt, pdftex]{amsart}
\usepackage[margin=1in]{geometry}
\usepackage[utf8]{inputenc}
\usepackage{tikz}
\usetikzlibrary{arrows.meta,decorations.pathreplacing,decorations.markings,shapes,calc}
\usetikzlibrary{matrix,arrows,decorations.pathmorphing,backgrounds,decorations.markings,positioning}
\usepackage{graphics}
\newcommand{\rotatesim}{\rotatebox[origin=c]{90}{$\sim$}}
\usepackage[all]{xy}
\usepackage{comment}
\usepackage{amsmath}
\usepackage{amssymb}
\usepackage{amsthm}
\usepackage{here}
\usepackage{amscd} 
\usepackage{tikz-cd}
\usepackage{mathrsfs}
\usepackage{mathtools}
\usepackage{mathabx}
\usepackage{scalefnt}
\usepackage{url}
\theoremstyle{plain}
\newtheorem{thm}{Theorem}[section]
\newtheorem{lem}[thm]{Lemma}
\newtheorem{cor}[thm]{Corollary}
\newtheorem{prop}[thm]{Proposition}

\theoremstyle{definition}

\newtheorem{rem}[thm]{Remark}

\numberwithin{equation}{section}
\def\A{{\mathbb A}}

\def\Q{{\mathbb Q}}
\def\R{{\mathbb R}}
\def\Z{{\mathbb Z}}
\def\C{{\mathbb C}}
\def\P{{\mathbb P}}
\def\B{{\mathbb B}}
\def\SS{{\mathfrak S}}
\def\id{\mathop{\mathrm{id}}\nolimits}

\def\Pic{\mathop{\mathrm{Pic}}\nolimits}

\def\SL{\mathop{\mathrm{SL}}\nolimits}

\def\dim{\mathop{\mathrm{dim}}\nolimits}
\def\div{\mathop{\mathrm{div}}\nolimits}

\def\Stab{\mathop{\mathrm{Stab}}\nolimits}

\def\diag{\mathop{\mathrm{diag}}\nolimits}
\def\antidiag{\mathop{\mathrm{antidiag}}\nolimits}

\def\ss{\mathop{\mathrm{ss}}}

\def\ord{\mathrm{ord}}
\def\GIT{\mathrm{GIT}}
\def\K{\mathrm{K}}
\def\tor{\mathrm{tor}}
\def\BB{\mathrm{BB}}

\def\L{\mathscr{L}}
\def\H{\mathscr{H}}
\def\g{\mathfrak{g}}

\def\M{\mathcal{M}}

\def\OO{\mathscr{O}}

\def\A{\mathcal{A}}
\def\D{\mathscr{D}}

\def\PP{\mathcal{P}}
\def\QQ{\mathcal{Q}}
\def\RR{\mathcal{R}}
\def\a{\alpha}
\def\b{\beta}
\def\g{\gamma}
\def\l{\langle}
\def\r{\rangle}

\def\exp{\mathop{\mathrm{exp}}\nolimits}

\def\SO{\mathop{\mathrm{SO}}\nolimits}

\def\PGL{\mathop{\mathrm{PGL}}\nolimits}
\def\U{\mathrm{U}}
\def\O{\mathrm{O}}

\allowdisplaybreaks[4]

\newcommand{\defeq}{\vcentcolon=}

\begin{document}

\title[Revisiting the moduli space of 8 points on $\mathbb{P}^1$]
{Revisiting the moduli space of 8 points on $\mathbb{P}^1$}

\author{Klaus Hulek}
\address{K.H: Institut f\"ur Algebraische Geometrie, Leibniz University Hannover, Welfengarten 1, 30060 Hannover, Germany}
\email{hulek@math.uni-hannover.de}
\author{Yota Maeda}
\address{Y.M: Department of Mathematics, Faculty of Science, Kyoto University, Kyoto 606-8502, Japan/Advanced Research Laboratory, Research Platform, Sony Group Corporation, 1-7-1 Konan, Minato-ku, Tokyo, 108-0075, Japan}
\email{y.maeda.math@gmail.com}

\date{\today}

\maketitle
\begin{abstract}
The moduli space of $8$ points on $\mathbb{P}^1$, a so-called ancestral Deligne-Mostow space, is, by work of Kond\={o}, also a moduli space of K3 surfaces. 
We prove that the Deligne-Mostow isomorphism does not lift to a morphism  between the  Kirwan blow-up of the GIT quotient  and the unique toroidal compactification of the corresponding ball quotient.
Moreover, we show that these spaces are not $K$-equivalent, even though they are natural blow-ups at the unique cusps and have the same cohomology.
This is analogous to the work of Casalaina-Martin-Grushevsky-Hulek-Laza on the moduli space of cubic surfaces. The moduli spaces of ordinary stable maps, that is the  Fulton-MacPherson compactification of the configuration space of points on $\mathbb{P}^1$,  play an important role in the proof. We further relate our computations to new developments in the minimal model program and recent work of Odaka.
We briefly discuss other cases of  moduli space of points on $\mathbb{P}^1$ where a similar behaviour can be observed, hinting at a more general, but not yet fully understood phenomenon.   
\end{abstract}

\section{Introduction}
It was shown by Casalaina-Martin-Grushevsky-Hulek-Laza \cite{CMGHL24} that the Kirwan blow-up and the toroidal compactification of the moduli space of (non-marked) smooth cubic surfaces are not isomorphic.
In this paper, we prove analogous results for the moduli space of unordered 8 points on $\P^1$, denoted by $\M^{\GIT}$.
The proof we give here is inspired by that of  \cite{CMGHL24}, but requires further ideas.
As we shall discuss in Section \ref{sec:othercases}, the behaviour observed here is shared 
by other ball quotients as well, thus pointing towards a much more general, and yet not fully understood, phenomenon.

The case of 8 points on $\P^1$ is of special interest for more than one reason. One is that it has more than one modular interpretation. Besides being a moduli space of points, it is, by work of Kond\={o}
\cite{Ko07a}, also closely related to moduli of K3 surfaces and automorphic forms. A further reason is that it is 
a so-called {\em ancestral} Deligne-Mostow variety in the sense of the discussion by Gallardo-Kerr-Schaffler \cite{GKS21}. This means that any Deligne-Mostow variety over the Gaussian integers
with arithmetic monodromy group, and which has cusps, can be embedded into this ball quotient. The other ancestral case is that of 12 points on $\P^1$, which plays the same role for the Eisenstein integers. In this paper, 
we shall concentrate on the Gaussian case and only briefly discuss the Eisenstein case, which is the contents of the subsequent paper  \cite{HKM24}.

\subsection{Main results}
The Deligne-Mostow theory \cite{DM86} gives us an isomorphism between $\M^{\GIT}$ and the Baily-Borel compactification of an appropriate $5$-dimensional ball quotient $\overline{\B^5/\Gamma}^{\BB}$.
We are interested in the lifting of the Deligne-Mostow isomorphism to the unique toroidal compactification.
There exist two natural blow-ups, playing important roles here: the Kirwan blow-up $f:\M^{\K}\to\M^{\GIT}$ and the toroidal compactification $\pi:\overline{\B^5/\Gamma}^{\tor}\to\overline{\B^5/\Gamma}^{\BB}$.
Here, the Kirwan blow-up $\M^{\K}$ is the partial desingularisation of $\M^{\GIT}$ whose centre is located in the polystable orbits (which is a unique point  $\{c_{4,4}\}$ in our case).
The toroidal compactification $\overline{\B^5/\Gamma}^{\tor}$ is a blow-up of $\overline{\B^5/\Gamma}^{\BB}$ at the point $\{\xi\}$, which is the unique cusp, i.e., the Baily-Borel boundary.
The above Deligne-Mostow isomoropshim sends $c_{4,4}$ to $\xi$, thus restricting to an isomorphism $\M^{\K}\setminus f^{-1}(c_{4,4})\cong \overline{\B^5/\Gamma}^{\tor}\setminus\pi^{-1}(\xi)$.
In this setting, our first main result asserts that the birational map $g:\M^{\K}\dashrightarrow\overline{\B^5/\Gamma}^{\tor}$ does not extend to a morphism.
\[\begin{tikzcd}
  \M^{\K} \ar[r, dotted, "g"] \arrow[d, "f"] & \overline{\B^5/\Gamma}^{\tor} \arrow[d, "\pi"] \\
  \M^{\GIT} \arrow[r, "\phi"] & \overline{\B^5/\Gamma}^{\BB}.
\end{tikzcd}\]

\begin{thm}[{Theorem \ref{thm:not_extend}}, Remark \ref{rem:newproof13}]
\label{mainthm:extendability}
Neither the Deligne-Mostow isomorphism $\phi:\M^{\GIT}\to\overline{\B^5/\Gamma}^{\BB}$ nor its inverse $\phi^{-1}$ lift to a morphism between the Kirwan blow-up $\M^{\K}$ and the unique toroidal compactification 
$\overline{\B^5/\Gamma}^{\tor}$.
\end{thm}

This result still leaves the possibility open that the Kirwan blow-up and the toroidal compactification are isomorphic as abstract varieties. 
One obstruction to this could be that the varieties are topologically different. Indeed, the topology of these varieties is of independent interest
(and indeed this was the starting point of \cite{CMGHL23} and  \cite{CMGHL24} in the case of cubic threefolds and cubic surfaces).
We compute the cohomology of these varieties, according to the Kirwan method \cite{Ki84, Ki85, Ki89} and Casalaina-Martin-Grushevsky-Hulek-Laza \cite{CMGHL23}.
Wherever a space $X$ has at most finite quotient singularities, we work with singular cohomology with rational coefficients and denote this by $H^k(X)$. In the other cases, notably the GIT quotient and 
the Baily-Borel compactification of ball quotients, we work with intersection cohomology (of middle perversity) and denote this by $IH^k(X)$. Note that for spaces with finite quotient singularities
singular cohomology and intersection cohomology coincide.  
The cohomology groups of the varieties under consideration are given as follows.
\begin{thm}[{Theorem \ref{thm:coh_previous_work}, \ref{thm:coh_ordered_tor}, \ref{thm:coh_M^K}, \ref{thm:coh_tor}}]
All the odd degree cohomology of the following projective varieties vanishes.
In even degrees, their Betti numbers are given by:
\begin{align*}
\renewcommand*{\arraystretch}{1.2}
\begin{array}{l|cccccc}
\hskip2cm j&0&2&4&6&8&10\\\hline
\dim H^j(\M^{\K})&1&2&3&3&2&1\\
\dim IH^j(\overline{\B^5/\Gamma}^{\BB})&1&1&2&2&1&1\\
\dim H^j(\overline{\B^5/\Gamma}^{\tor})&1&2&3&3&2&1\\
\dim H^j(\M_{\ord}^{\K})&1&43&99&99&43&1\\
\dim IH^j(\overline{\B^5/\Gamma_{\ord}}^{\BB})&1&8&29&29&8&1\\
\dim IH^j(\overline{\B^5/\Gamma_{\ord}}^{\tor})&1&43&99&99&43&1
\end{array}
\end{align*}
thus, all the Betti numbers of $\M^{\K}$ and $\overline{\B^5/\Gamma}^{\tor}$ are the same.
\end{thm}
Here, $\overline{\B^5/\Gamma_{\ord}}^{\BB}$ denotes the Baily-Borel compactification of a $5$-dimensional ball quotient, which is 
 an $\SS_8$-cover of $\overline{\B^5/\Gamma}^{\BB}$ and isomorphic to $\M^{\GIT}_{\ord}$, the moduli space of ordered 8 points on $\P^1$.
Also, we denote by $\M_{\ord}^{\K}$ the Kirwan blow-up of $\M^{\GIT}_{\ord}$ and by $\overline{\B^5/\Gamma_{\ord}}^{\tor}$ the toroidal blow-up of $\overline{\B^5/\Gamma_{\ord}}^{\BB}$.
For more precise descriptions of these varieties,  as well as the bounded symmetric domain and arithmetic subgroups, see Section \ref{section:preparation}.

Again, this result leaves the possibility that $\M^{\K}$ and $\overline{\B^5/\Gamma}^{\tor}$ are isomorphic as abstract varieties. We  rule this out by showing 
that these spaces are not $K$-equivalent. Recall that  two projective normal $\Q$-Gorenstein varieties $X$ and $Y$ are called \textit{$K$-equivalent} if there  is a common
resolution of singularities  $Z$ dominating $X$ and $Y$ birationally
\begin{equation*}
\xymatrix{
&Z\ar[ld]_{f_X} \ar[rd]^{f_Y}&\\
X\ar@{<-->}[rr]&&Y
}
\end{equation*}
such that $f_X^*K_X \sim_{\Q} f_Y^*K_Y$. 
For $K$-equivalent varieties, the top intersection numbers are equal: $K_X^n=K_Y^n$, where $n$ is the dimension of $X$ and $Y$. We shall use this property to show that 
$\M^{\K}$ and $\overline{\B^5/\Gamma}^{\tor}$ are not $K$-equivalent.  
Thus, these varieties are in particular not isomorphic as abstract varieties,  even though they are the blow-ups at the same points of $\M^{\GIT}\ \cong \overline{\B^5/\Gamma}^{\BB}$ and 
have the same Betti numbers.

\begin{thm}[{Theorem \ref{thm:not_K_equiv}}]
\label{mainthm:not_K_equiv}
The Kirwan blow-up $\M^{\K}$ and the toroidal compactification $\overline{\B^5/\Gamma}^{\tor}$ are not $K$-equivalent and hence, in particular, not isomorphic as abstract varieties.
\end{thm}

\begin{rem}
\label{rem:odaka}
We can interpret Theorems \ref{mainthm:extendability} and \ref{mainthm:not_K_equiv} in the context of the  minimal model program and semi-toric compactifications.
This gives, in particular, an independent proof of one direction of Theorem \ref{mainthm:extendability} using the characterization of (semi-)toric compactifications. This proof does not require the calculation of the second Betti number of these spaces, but it uses the Luna slice calculations. We refer the reader to Subsection  \ref{subsection:mmp} for more details.  
\end{rem}

 As we shall see later, the situation is in contrast to the case of the moduli space of ordered points, where we have an isomorphism 
$\M_{\ord}^{\K} \cong \overline{\B^5/\Gamma_{\ord}}^{\tor}$.

\begin{rem}
    \label{rem:Rapoport}
    Kudla-Rapoport \cite{KR12} studied the descent problem of Deligne-Mostow isomorphisms and ball quotients at the level of moduli stacks over the natural field of definitions.
    In particular, in the case of 12 points, they gave an interpretation as a DM-stack parameterizing abelian varieties, showing the Deligne-Mostow isomorphism comes from a morphism between DM-stacks and its image is a complement of Kudla-Rapoport cycles \cite[Theorem 8.1]{KR12}.
   It seems interesting to ask whether a similar result holds in our (and similar) situations.
\end{rem}

\subsection{Outline of the proof of Theorem \ref{mainthm:extendability}}
The strategy of the proof of Theorem \ref{mainthm:extendability} is as follows.
As in \cite{CMGHL24} the argument is divided into two steps.
We first prove that the discriminant divisor and the boundary divisor intersect non-transversally in the Kirwan blow-up.
This is done in terms of a local computation by using the Luna slice. Secondly, we show that the corresponding divisors intersect generically transversally in the 
toroidal compactification of  the 5-dimensional ball quotient.
Here is a major difference to \cite{CMGHL24}.
This is because we cannot use Naruki's compactification.
Instead, we work on a  sequence of blow-ups of the Baily-Borel compactification of the 5-dimensional ball quotient.
This was studied in detail in \cite{GKS21, KM11} and can be described in terms of moduli spaces of weighted pointed stable curves \cite{Ha03}.
The discriminant divisor and boundary divisor exist as normal crossing divisors in these spaces, thus we can use this to prove the generic transversality of the divisors in the toroidal compactification.

\subsection{Organization of the paper}
In Section \ref{section:preparation}, we describe the relationship between GIT quotients and ball quotients.
In Section \ref{section:extendability}, we prove Theorem \ref{mainthm:extendability} through local computations.
In Section \ref{section:canonical_bundles}, we compute the top self-intersection number of canonical bundles, deduce Theorem \ref{mainthm:not_K_equiv} and discuss the relation to the minimal model program.
In Section \ref{sec:Bettinumbers}, we compute the cohomology by using the Kirwan method.
In Section \ref{sec:othercases}, we will briefly discuss other Deligne-Mostow varieties.

\subsection*{Acknowledgements}
The authors wish to express their thanks to Sebastian  Casalaina-Martin and Sigeyuki Kond\={o} for helpful discussions.
We would also like to thank Yuji Odaka for valuable information regarding semi-toric compactifications concerning Subsection \ref{subsection:mmp}.
We are grateful to Michael Rapoport for pointing out the connection with his work with Kudla about occult period maps.
The second author would like to thank Masafumi Hattori and Takuya Yamauchi for their comments, and Kanazawa University and Leibniz University Hannover for their hospitality.
The first author is partially supported by the DFG grant Hu 337/7-2 and the second author is supported by JST ACT-X JPMJAX200P.

\section{GIT and ball quotients}
\label{section:preparation}

Below, we consider the moduli spaces of \textit{ordered} and \textit{unordered} 8 points on $\P^1$.
Throughout this paper, the phrase ``8 points on $\P^1$" will always  mean ``unordered 8 points on $\P^1$" for simplicity.
Let
\[\M_{\ord}^{\GIT}\defeq (\P^1)^8//\SL_2(\C),\quad \M^{\GIT}\defeq \P^8//\SL_2(\C).\]
Here, the GIT quotients are taken with respect to the symmetric linearisation $\OO(1,\cdots, 1)$ and $\OO(1)$. We also note, see \cite[Theorem 1.1]{KM11}, that  
\[
\M_{\ord}^{\GIT}/\SS_8 \cong  \M^{\GIT}.
\]
We denote by $\varphi_1:\M_{\ord}^{\K}\to\M_{\ord}^{\GIT}$ and $f:\M^{\K}\to\M^{\GIT}$ the Kirwan blow-ups \cite{Ki85}.

 As in \cite{Ko07a}, 
 we consider the free  $\Z[\sqrt{-1}]$-module of rank $2$ equipped with the Hermitian forms defined by the following matrices
 \[\begin{pmatrix}
0 & 1+\sqrt{-1} \\
1-\sqrt{-1} & 0 \\
\end{pmatrix},\quad
 \begin{pmatrix}
-2 & 1-\sqrt{-1} \\
1+\sqrt{-1} & -2 \\
\end{pmatrix}.
\]
isomorphic to $U \oplus U(2)$ and $D_4(-1)$, respectively, 
where $U$ denotes the hyperbolic plane, $U(2)$ is the hyperbolic plane where the form has been 
multiplied by $2$ and $D_4(-1)$ is the negative $D_4$-lattice. By abuse of notation, we will also denote the Hermitian lattices by these symbols.

Here, let $L\defeq U\oplus U(2)\oplus D_4(-1)^{\oplus 2}$ be the Hermitian lattice of signature $(1,5)$ over $\Z[\sqrt{-1}]$, defined by the above Hermitian forms.
Let $\U(L)$ be the unitary group scheme over $\Z$ preserving the lattice $L$, and $\Gamma\defeq\U(L)(\Z)$.
Now, there is the Hermitian symmetric domain $\B^5$ associated with the reductive group $\U(L)(\R)\cong \U(1,5)$ defined by 
\[\B^5\defeq\{v\in L\otimes_{\Z[\sqrt{-1}]}\C\mid \l v,v\r>0 \}/\C^{\times}\]
which is isomorphic to the $5$-dimensional complex ball.
Let $L^{\vee}$ be the \textit{dual lattice} of $L$, which contains $L$ as a finite $\Z[\sqrt{-1}]$-module, and $A_L\defeq L^{\vee}/L$ be the \textit{discriminant group}, isomorphic to $\left(\Z[\sqrt{-1}]/(1+\sqrt{-1})\Z[\sqrt{-1}]\right)^6$ in this situation. 
Now, let us introduce an important arithmetic subgroup $\Gamma_{\ord}\subset\Gamma$, which is called the \textit{discriminant kernel}:
\[\Gamma_{\ord}\defeq\{g\in\Gamma\mid g(v)\equiv v\bmod L\ (\forall v\in A_L)\}.\]
This data gives us the notion of the \textit{ball quotients} 
$\B^5/\Gamma_{\ord}$ and $\B^5/\Gamma$
which are quasi-projective varieties over $\C$.
We denote by $\overline{\B^5/\Gamma_{\ord}}^{\BB}$ and  $\overline{\B^5/\Gamma}^{\BB}$ (resp. $\overline{\B^5/\Gamma_{\ord}}^{\tor}$ and $\overline{\B^5/\Gamma}^{\tor}$) the Baily-Borel compactifications (resp. toroidal compactifications) of the corresponding ball quotients.
Note that the toroidal compactifications of ball quotients are canonical as there is no choice of a fan involved.
Further, let 
\[H\defeq\bigcup_{\l\ell,\ell\r=-2} H(\ell)\]
be the discriminant divisor where 
\[H(\ell)= \{v\in\B^5\mid\l v,\ell\r=0\}\]
is the special divisor
with respect to a root $\ell\in L$, see \cite[Subsection 3.4]{Ko07a}.

Next, we describe the stable, semi-stable and polystable loci on $\M_{\ord}^{\GIT}$ and $\M^{\GIT}$. This goes back to very classical results of GIT, in fact 
Mumford's seminal work, see  \cite[Chapter 4, \S 2]{MFK94}. 
In our cases, this is spelled out as follows. In the ordered case, 8 points define a stable (resp. semi-stable) GIT-point if and only if no 4 points (resp. 5 points) coincide, see also
\cite[Subsection 4.4]{Ko07a} or \cite[Example 2, p31]{Do88}.
Polystable points (that is, strictly semi-stable points whose orbit is closed) correspond to the points $(4,4)$, which means that we have two different points, each with multiplicity $4$; 
for the notation, see \cite[Subsection 4.4]{Ko07a}.
In the unordered case, stable, semi-stable and polystable points are described in the same way as above, see also \cite[Subsection 7.2 (c)]{Mu03}. 

A crucial result of  Kond\={o},  \cite[Theorem 4.6]{Ko07a}, says that there are $\SS_8$-equivariant isomorphisms
\begin{align*}
    \phi_{\ord} &:\M^{\GIT}_{\ord}\xrightarrow{\sim} \overline{\B^5/\Gamma_{\ord}}^{\BB}\\
    \phi &:\M^{\GIT} \xrightarrow{\sim} \overline{\B^5/\Gamma}^{\BB},
\end{align*}
where the second isomorphism goes back to \cite{DM86}. 

These isomorphisms also allow us to describe the subloci of $8$-tuples consisting of different points, the discriminant locus of stable, but not distinct, 8-tuples and the properly polystable loci.
For this, let $(\M^{\GIT}_{\ord})^o\subset \M^{\GIT}_{\ord}$ (resp. $(\M^{\GIT})^o\subset \M^{\GIT}$) 
be the moduli space of distinct ordered 8 points on $\P^1$ (resp. the moduli space of distinct 8 points on $\P^1$).
By \cite[Theorem 3.3]{Ko07a}, the morphisms $  \phi_{\ord}$ and $\phi $ restrict to isomorphisms:
\begin{align*}
    \phi_{\ord}\vert_{(\M^{\GIT}_{\ord})^o}&:(\M^{\GIT}_{\ord})^o\  \xrightarrow{\sim} (\B^5\setminus H)/\Gamma_{\ord}\\
    \phi\vert_{(\M^{\GIT})^o}&:(\M^{\GIT})^o \xrightarrow{\sim} (\B^5\setminus H)/\Gamma.
\end{align*}

Also the isomorphisms   $\phi_{\ord}$ and $\phi$ identify the discriminant locus of stable, but not distinct 8 points on $\M_{\ord}^{\GIT}$ and $\M^{\GIT}$ with $H/\Gamma_{\ord}$ and 
$H/\Gamma$  respectively. 
It turns out that the discriminant divisor $H/\Gamma_{\ord}$ has 28 irreducible components, whereas $H/\Gamma$ is irreducible.
See also \cite[Subsection 4.2]{Ko07a}, asserting that $A_L$ contains 64 vectors: 1 zero vector, 35 isotropic vectors and 28 non-isotropic vectors.

Finally, the properly polystable points are identified with the cusps of the Borel compactification, namely 
 $\overline{(\B^5/\Gamma_{\ord}})^{\BB}\setminus(\B^5/\Gamma_{\ord})$ and $\overline{(\B^5/\Gamma})^{\BB}\setminus(\B^5/\Gamma$)) respectively.
There are 35 cusps on $\overline{\B^5/\Gamma_{\ord}}^{\BB}$ (also corresponding to the 35 isotropic vectors in $A_L$), but $\overline{\B^5/\Gamma}^{\BB}$ has a unique cusp.
This directly follows from \cite[Subsection 4.2, Proposition 4.4]{Ko07a}, but we will see this in detail when we study the blow-up sequences.

The moduli spaces under consideration are also closely related to moduli spaces of stable curves. We do not repeat all details of the general theory here, but 
recall some notions as they are relevant for our purposes.
Let $\overline{\M}_{0,8(\frac{1}{4}+\epsilon)}$ be the smooth projective  variety which is the coarse moduli space representing the moduli problem of weighted pointed stable curves of type $(0,8(\frac{1}{4}+\epsilon))$ with $0<\epsilon \ll 1$ in the sense of \cite[Theorem 2.1]{Ha03} or \cite[Definition 2.1, Theorem 2.2]{KM11}, see also \cite[Lemma 2.3, Remark 2.4, Remark 2.11, Example 2.12]{GKS21}.
This is also realized as the KSBA compactification \cite[Subsection 3.2]{GKS21}.
$\overline{\M}_{0,8}$ is defined in the same way, but in this case, this is exactly the GIT quotient of $\P^1[8]$, the Fulton-MacPherson compactification of the configuration space of 8 points on $\P^1$ \cite{FM94}, by $\SL_2$; see also \cite[p55]{MM07}.
More generally, this is interpreted as the wonderful compactification \cite[p536, Subsection 4.2]{Li09} (or the Deligne-Mumford compactification \cite[Remark 2.9]{GKS21}).

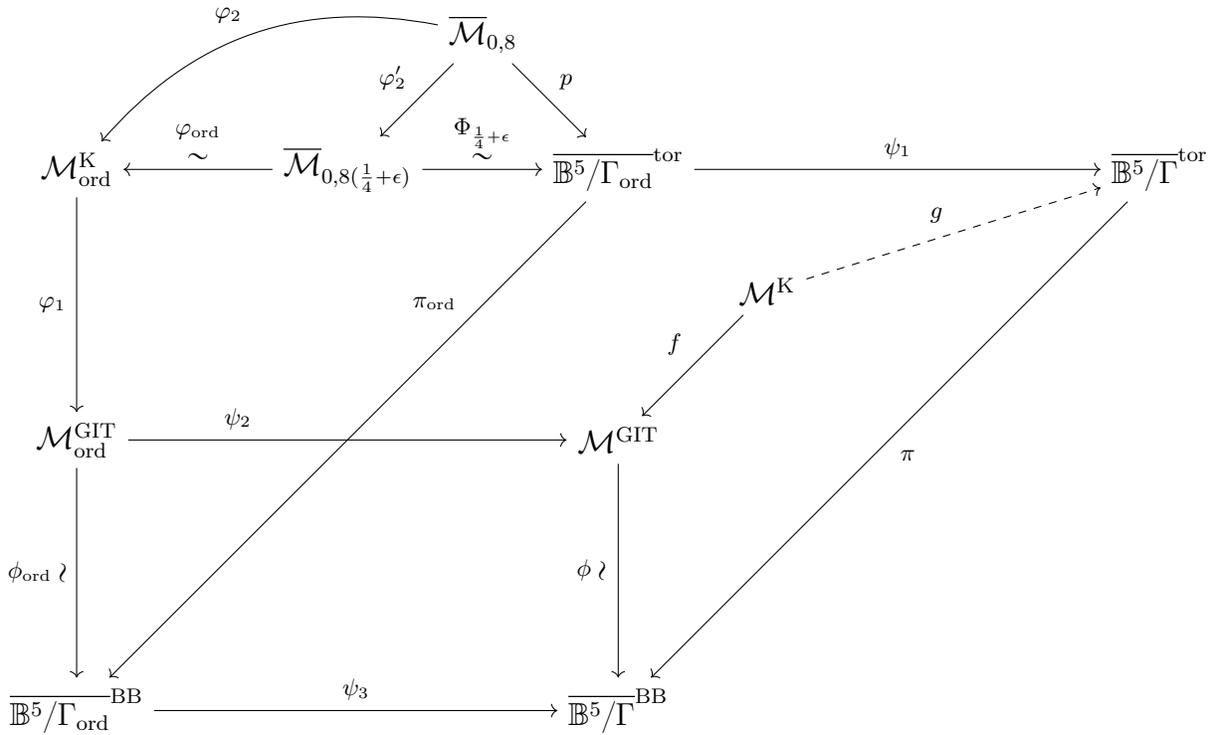
\begin{figure}[h]
\centering
\[
  \begin{tikzpicture}[
    labelsize/.style={font=\scriptsize},
    isolabelsize/.style={font=\normalsize},
    hom/.style={->,auto,labelsize},
    scale=0.9,
  ]
  
  \node(K) at (0,8) {$\M_{\ord}^{\K}$};
  \node(Qu) at (4,8) {$\overline{\M}_{0,8(\frac{1}{4}+\epsilon)}$};
  \node(ctor) at (8,8) {$\overline{\B^5/\Gamma_{\ord}}^{\tor}$}; 
  \node(rtor) at (16,8) {$\overline{\B^5/\Gamma}^{\tor}$}; 
  \node(git) at (0,4) {$\M^{\GIT}_{\ord}$};
  \node(gitl) at (8,4) {$\M^{\GIT}$};
  \node(lbb) at (0,0) {$\overline{\B^5/\Gamma_{\ord}}^{\BB}$}; 
  \node(cbb) at (8,0) {$\overline{\B^5/\Gamma}^{\BB}$}; 
  \node(k) at (10.2,6.2) {$\M^{\K}$};
  \node(s) at (6,10) {$\overline{\M}_{0,8}$};
  
  \draw[hom,swap] (Qu) to node[align=center,yshift=-2pt]{
    {\scriptsize $\varphi_{\ord}$}\\
    {\normalsize $\sim$}
  } (K);
  \draw[hom] (Qu) to node[align=center,yshift=-2pt]{
    {\scriptsize $\Phi_{\frac{1}{4}+\epsilon}$}\\
    {\normalsize $\sim$}
  } (ctor);
  \draw[hom,swap] (git) to  node[align=center,xshift=5pt]{
    {\scriptsize $\phi_{\ord}$}
    {\normalsize$\rotatesim$}
  } (lbb);
  \draw[hom,swap] (gitl) to node[align=center,xshift=5pt]{
    {\scriptsize $\phi$}
    \normalsize$\rotatesim$
  } (cbb);
  
  \draw[hom] (ctor) to node{$\psi_1$} (rtor);
  \draw[hom,near start] (git) to node{$\psi_2$} (gitl);
  \draw[hom] (lbb) to node{$\psi_3$} (cbb);
  \draw[hom,swap] (K) to node{$\varphi_1$} (git);
  \draw[hom,swap] (s) to node{$\varphi_2'$} (Qu);
  \draw[hom] (s) to node{$p$} (ctor);
  \draw[hom,swap] (k) to node{$f$} (gitl);
  \draw[hom,dashed] (k) to node{$g$} (rtor);
  \draw[hom,near start,swap] (ctor) to node{$\pi_{\ord}$} (lbb);
  \draw[hom] (rtor) to node{$\pi$} (cbb);
  \draw[hom,bend right,swap](s) to node{$\varphi_2$}(K);
  
  \end{tikzpicture}
\]
\caption{Relationship between several compactifications}
\label{com_diag}
\end{figure}

We describe the relation of these spaces in Figure \ref{com_diag}.

\begin{enumerate}
    \item $\psi_i$ is a morphism by the above discussion about stable conditions for $i=1,2,3$.
    \item $\phi$ is an isomorphism \cite{DM86}.
    \item $\phi_{\ord}$ is an $\SS_8$-equivalent isomorphism \cite[Theorem 4.6]{Ko07a}.
    \item $\varphi_{\ord}$ is an isomorphism \cite[Theorem 1.1]{KM11}.
    \item $\Phi_{\frac{1}{4}+\epsilon}$ is an isomorphism \cite[Theorem 1.1]{GKS21}.
    \item $p$ is a morphism \cite[Proposition 2.13]{GKS21}.
    \item The blow-up sequences $\varphi_1, \varphi_2$ are considered in \cite[Theorem 4.1 (i), (iii)]{KM11}.
    About the contraction of divisors of these morphisms, see \cite[Proposition 4.5]{Ha03} or \cite[p1121]{KM11}.
    We study these morphisms in detail in Subsection \ref{transversality}.
    \item $\overline{\M}_{0,8}$ is a normal crossing compactification of $(\B^5\setminus H)/\Gamma_{\ord}$, see \cite[p345]{Ha03} or \cite[Proposition 2.13]{GKS21}.
    \item $\M_{\ord}^{\K}\cong\overline{\M}_{0,8(\frac{1}{4}+\epsilon)}$ is nonsingular \cite[Section 4]{KM11}.
\end{enumerate}

We conclude this section with a remark about the toroidal boundary, which is defined by
$T_{\ord}\defeq\overline{(\B^5/\Gamma_{\ord})}^{\tor}\setminus(\B^5/\Gamma_{\ord})$ and $T\defeq\overline{(\B^5/\Gamma)}^{\tor}\setminus(\B^5/\Gamma)$ 
respectively.
The divisor $T_{\ord}$ has 35 irreducible components (mapping to the 35 cusps in the Baily-Borel compactification). We write them as $T_{\ord, i}$ for $i=1,\dots 35$.
Note that $T_{\ord, i}\cong\P^2\times\P^2$ by \cite[Remark 6]{MS21} or \cite[Example 2.12]{GKS21}.
The boundary divisor $T$ is irreducible (and maps to the unique cusp in the Baily-Borel compactification); see also \cite[Proposition 4.7]{Ko07a}.
We study $T_{\ord}$ and $T$ in detail in Lemma \ref{lem:boundary}.

\section{(Non-) Extendability of the Deligne-Mostow isomorphism}
\label{section:extendability}
\subsection{Non-transversality in the Kirwan blow-up}
\label{non-transversality}
In this subsection, we show that the discriminant divisor and the boundary divisor do not intersect transversally in $\M^{\K}$.

To prove this statement, we will need a detailed analysis of stabilizer groups. For an algebraic group $G$ we will denote the connected component of the identity by $G^{\circ}$. 
The following two lemmas are modeled on \cite[Lemma 2.3]{CMGHL24} and \cite[Lemma 2.4]{CMGHL24}.
Below, we denote by $x_0, x_1$ the homogeneous coordinate of $\P^1$.
In this terminology, the polystable point $c_{4,4}$ corresponds to $x_0^4x_1^4$.
\begin{lem}
\label{lem:stabilizers}
The following equalities hold:
\begin{align*} 
    R&\defeq\Stab(c_{4,4})=\Bigl\{\begin{pmatrix}
\lambda & 0 \\
0 & \lambda^{-1} \\
\end{pmatrix}\in\SL_2(\C)\Bigr\}\bigcup \Bigl\{\begin{pmatrix}
0 & \lambda \\
-\lambda^{-1} & 0 \\
\end{pmatrix}\in\SL_2(\C)\Bigr\}\cong \C^{\times}\rtimes\SS_2\\
    R^{\circ}&\defeq\Stab(c_{4,4})^{\circ}\cong\C^{\times}.
    \end{align*}    
\end{lem}

Now, let us prepare for the local computations.
The Luna slice theorem gives us a tool to study them as handled in the case of the moduli space of cubic threefolds \cite[Subsection 4.3.1]{CMGHL23} or cubic surfaces \cite[Lemma 3.4]{CMGHL24}; see also \cite[Subsection 7.1]{Zh05}.

\begin{lem}
\label{lem:Luna_slice}
A Luna slice for $c_{4,4}$, normal to the orbit $\SL_2(\C)\cdot \{c_{4,4}\}\subset\P^8$, 
is isomorphic to $\C^6$, spanned by the 6 monomials 
\[x_0^8, \quad x_1^8,\quad x_0^7x_1,\quad x_0x_1^7,\quad x_0^6x_1^2,\quad x_0^2x_1^6\]
in the tangent space $H^0(\P^1, \OO_{\P^1}(8))$.
Projectively, 
\begin{align*}
    \P^6&=\{\a_0x_0^8+\a_1x_1^8+\b_0x_0^7x_1+\b_1 x_0x_1^7+\g_0x_0^6x_1^2+\g_1x_0^2x_1^6+kx_0^4x_1^4\}\\
    &\subset \P H^0(\P^1,\OO_{\P^1}(8)) =\P^8.
\end{align*}
\end{lem}
\begin{proof}
This can be proven in the same way as \cite[Subsection 4.3.1]{CMGHL23}.
We note that the (affine) tangent space of the orbit  is given by the entries of the matrix
\[\begin{pmatrix}
x_0^4x_1^4 & x_0^3x_1^5 \\
x_0^5x_1^3 & x_0^4x_1^4 \\
\end{pmatrix}.
\]
\end{proof}
Let \[\diag(\lambda,\lambda^{-1})\defeq
\begin{pmatrix}
\lambda & 0 \\
0 & \lambda^{-1} \\
\end{pmatrix},\ 
\antidiag(\lambda,-\lambda^{-1})\defeq
\begin{pmatrix}
0 & \lambda \\
-\lambda^{-1} & 0 \\
\end{pmatrix}.
\]
Then, the action of an element of $\Stab(c_{4,4})$ is given by  
\begin{align}
\label{eq:action_diag}
    \diag(\lambda,\lambda^{-1})\cdot(\a_0,\a_1,\b_0,\b_1,\g_0,\g_1)&=(\lambda^8\a_0,\lambda^{-8}\a_1,\lambda^6\b_0,\lambda^{-6}\b_1,\lambda^4\g_0,\lambda^{-4}\g_1)\\
\label{eq:action_anti_diag}
    \antidiag(\lambda,-\lambda^{-1})\cdot(\a_0,\a_1,\b_0,\b_1,\g_0,\g_1)&=(\lambda^{-8}\a_1,\lambda^8\a_0, -\lambda^{-6}\b_1,-\lambda^6\b_0,\lambda^{-4}\g_1,\lambda^4\g_0).
\end{align}
We write the coordinates of the Kirwan blow-up $Bl_0\C^6\subset\C^6\times\P^5$ of the Luna slice as $(\a_0, \a_1, \b_0, \b_1, \g_0, \g_1)$ and $[S_0:S_1:T_0:T_1:U_0:U_1]$.

\begin{lem}
\label{lem:unstable_locus}
The unstable locus of the action of the stabilizer $\SL(c_{4,4})$ of $c_{4,4}$ in $\SL_2(\C)$ 
is the codimension three locus
\[\{S_0=T_0=U_0=0\}\cup\{S_1=T_1=U_1=0\}\subset\P^5.\]
\end{lem}
\begin{proof}
From (\ref{eq:action_diag}), the action of $R^{\circ}\cong\C^{\times}$ is given by 
\begin{align*}
    \diag(\lambda,\lambda^{-1})\cdot (S_0,S_1,T_0,T_1,U_0,U_1)=(\lambda^8S_0,\lambda^{-8}S_1,\lambda^6T_0,\lambda^{-6}T_1,\lambda^4U_0,\lambda^{-4}U_1).
\end{align*}
Thus, the representation of  $\C^{\times}$ on $\C^6$ decomposes into 6 characters.
By the same discussion as in the proof of \cite[Lemma 3.6]{CMGHL24},  the points in the unstable locus are characterized 
by the property that the convex hull spanned by the weights appearing in the above representation does not contain the origin.
This condition holds if and only if $\{S_0=T_0=U_0=0\}$ or $\{S_1=T_1=U_1=0\}$.
\end{proof}

We denote by $\D_{\ord}$ (resp. $\D$) the discriminant divisor, corresponding to the closure of $H/\Gamma_{\ord}$ (resp. $H/\Gamma$), through the isomorphism $\phi_{\ord}:\M_{\ord}^{\GIT}\to\overline{\B^5/\Gamma_{\ord}}^{\BB}$ (resp. $\phi:\M^{\GIT}\to\overline{\B^5/\Gamma}^{\BB}$).
Let $\widetilde{\D}$ be the strict transform of the discriminant divisor $\D$ in the blow-up $\M^{\K}\to\M^{\GIT}$.
Besides, let $\Delta_{\ord}$ (resp. $\Delta$) be the union of boundary divisors of $\M^{\K}_{\ord}$ (resp. $\M^{\K}$).

\begin{thm}
\label{thm:nonord_nontransversal}
The strict transform $\widetilde{\D}$ and the boundary divisor $\Delta$ 
do not meet generically transversally in $\M^{\K}$.
\end{thm}
\begin{proof}
We work on the local computation via the Luna slice described in Lemma \ref{lem:Luna_slice}.
Before taking the GIT quotient, we have the blow-up
\[Bl_0\C^6\to\C^6,\]
where the coordinates of the affine space (the Luna slice) $\C^6$ are $\a_0, \a_1, \b_0, \b_1, \g_0, \g_1$
(this is the first step of the Kirwan blow-up).
In this Luna slice, $\D$ locally near the origin, is given by 
\begin{align*}
    &\mathrm{disc}(x^4+\a_0x^2+\b_0x+\g_0)\cdot\mathrm{disc}(y^4+\a_1y^2+\b_1y+\g_1)\\
    &=(256\g_0^3-128\a_0^2\g_0^2+144\a_0\b_0^2\g_0-27\b_0^4+16\a_0^4\g_0-4\a_0^3\b_0^2)\\
    &\cdot (256\g_1^3-128\a_1^2\g_1^2+144\a_1\b_1^2\g_1-27\b_1^4+16\a_1^4\g_1-4\a_1^3\b_1^2)\\
    &=0.
\end{align*}
The reason for this is that we consider the polystable point given by $x^4y^4$ and that the versal deformation of the quadruple point $x^4=0$ is given by $x^4+\a_0x^2+\b_0x+\g_0=0$.
We write this as  $V\defeq V_1\cup V_2$ with $\SS_2$ permuting the two components.
We consider the affine loci 
\[\PP\defeq(S_0\neq 0),\quad \QQ\defeq(T_0\neq 0),\quad \RR\defeq(U_0\neq 0).\]

First, on $\PP$, the inverse image of $V$ is 
\begin{align}
    &\a_0^6(256u_0^3-128\a_0u_0^2+144\a_0t_0^2u_0-27\a_0t_0^4+16\a_0^2u_0-4\a_0^2t_0^2)\label{eq:Luna_slice_P}\\
    &\cdot(256u_1^3-128\a_0s_1^2u_1^2+144\a_0s_1t_1^2u_1-27\a_0t_1^4+16\a_0^2s_1^4u_1-4\a_0^2s_1^3t_1^2)\notag\\
    &=0,\notag
\end{align}
where
\[s_1\defeq\frac{S_1}{S_0},\quad t_i\defeq\frac{T_i}{S_0},\quad u_i\defeq\frac{U_i}{S_0}\]
and the coordinates of $\PP$ are $(\a_0,s_1,t_0,t_1,u_0,u_1)$.
Hence, the strict transform of $V$ is given by 
\begin{align*}
    &(256u_0^3-128\a_0u_0^2+144\a_0t_0^2u_0-27\a_0t_0^4+16\a_0^2u_0-4\a_0^2t_0^2)\\
    &\cdot(256u_1^3-128\a_0s_1^2u_1^2+144\a_0s_1t_1^2u_1-27\a_0t_1^4+16\a_0^2s_1^4u_1-4\a_0^2s_1^3t_1^2)\\
    &=0,
\end{align*}
since the exceptional divisor of the blow-up is $(\a_0=0)$.
The Luna slice for the action $\mathbb{T}\subset R$ is given by $(s_1=1)$ in $\PP$ because for any point $(\a_0,s_1,t_0,t_1,u_0,u_1)\in\PP$ with $s_1\neq 0$, there exists a complex number $\lambda$ such that $\lambda^{-16}=s_1$.
Thus, the intersection of the strict transform of $V$ with this Luna slice is given by 
\begin{align*}
    &\{256u_0^3-\a_0(128u_0^2+144t_0^2u_0-27t_0^4+16\a_0u_0-4\a_0t_0^2)\}\\
    &\cdot\{256u_1^3-\a_0(128u_1^2+144t_1^2u_1-27t_1^4+16\a_0u_1-4\a_0t_1^2)\}\\
    &=0.
\end{align*}
This shows that the first (resp. second) factor intersect the exceptional divisor $(\a_0=0)$ non-transversally along $(u_0=0)$ (resp. $(u_1=0)$).

Next, on $\QQ$, the inverse image of $V$ is 
\begin{align*}
    &\b_0^6(256u_0^3-128\b_0s_0^2u_0^2+144\b_0s_0u_0-27\b_0+16\b_0^2s_0u_0-4\b_0^2s_0)\\
    &\cdot(256u_1^3-128\b_0s_1^2u_1^2+144\b_0s_1t_1^2u_1-27\b_0t_1^4+16\b_0^2s_1^4u_1-4\b_0^2\a_1^3u_1^2)\\
    &=0,
\end{align*}
where
\[s_i\defeq\frac{S_i}{T_0},\quad t_1\defeq\frac{T_1}{T_0},\quad u_i\defeq\frac{U_i}{T_0}\]
and the coordinates of $\PP$ is $(s_0,s_1,\b_0,t_1,u_0,u_1)$.
Hence, the strict transform of $V$ is given by 
\begin{align*}
    &(256u_0^3-128\b_0s_0^2u_0^2+144\b_0s_0u_0-27\b_0+16\b_0^2s_0u_0-4\b_0^2s_0)\\
    &\cdot(256u_1^3-128\b_0s_1^2u_1^2+144\b_0s_1t_1^2u_1-27\b_0t_1^4+16\b_0^2s_1^4u_1-4\b_0^2\a_1^3u_1^2)\\
    &=0,
\end{align*}
since the exceptional divisor of the blow-up is $(\b_0=0)$.
The Luna slice for the action $\mathbb{T}\subset R$ is given by $(t_1=1)$ in $\PP$ because for any point $(s_0,s_1,\b_0,t_1,u_0,u_1)\in\QQ$ with $t_1\neq 0$, there exists a complex number $\lambda$ such that $\lambda^{-12}=t_1$.
Thus, the intersection of the strict transform of $V$ with this Luna slice is given by 
\begin{align*}
    &\{256u_0^3-\b_0(128s_0^2u_0^2+144s_0u_0-27+16\b_0s_0u_0-4\b_0s_0)\}\\
    &\cdot\{256u_1^3-\b_0(128s_1^2u_1^2+144s_1u_1-27+16\b_0s_1^4u_1-4\b_0\a_1^3u_1^2)\}\\
    &=0.
\end{align*}
This shows that the first (resp. second) factor intersect the exceptional divisor $(\b_0=0)$ non-transversally along $(u_0=0)$ (resp. $(u_1=0)$).

Finally, on $\RR$, the inverse image of $V$ is 
\begin{align*}
    &\g_0^6(256-128\g_0s_0^2+144\g_0s_0t_0^2-27\g_0t_0^4+16\g_0^2s_0^4-4\g_0^2s_0^3t_0^2)\\
    &\cdot(256u_1^3-128\g_0s_1^2+144\g_0s_1t_1^2-27\g_0t_1^4+16\g_0^2s_1^4-4\g_0^2s_1^3t_1^2)\\
    &=0,
\end{align*}
where
\[s_i\defeq\frac{S_i}{U_0},\quad t_i\defeq\frac{T_i}{U_0},\quad u_1\defeq\frac{U_1}{U_0}\]
and the coordinates of $\RR$ are $(s_0,s_1,t_0,t_1,\g_0,u_1)$.
Hence the strict transform of $V$ is given by 
\begin{align*}
    &(256-128\g_0s_0^2+144\g_0s_0t_0^2-27\g_0t_0^4+16\g_0^2s_0^4-4\g_0^2s_0^3t_0^2)\\
    &\cdot(256u_1^3-128\g_0s_1^2+144\g_0s_1t_1^2-27\g_0t_1^4+16\g_0^2s_1^4-4\g_0^2s_1^3t_1^2)\\
    &=0,
\end{align*}
since the exceptional divisor of the blow-up is $(\g_0=0)$.
The Luna slice for the action $\mathbb{T}\subset R$ is given by $(g_1=1)$ in $\RR$ because for any point $(s_0,s_1,t_0,t_1,\g_0,u_1)\in\RR$ with $u_1\neq 0$, there exists a complex number $\lambda$ such that $\lambda^{-8}=\g_1$.
Thus, the intersection of the strict transform of $V$ with this Luna slice is given by 
\begin{align*}
    &(256-128\g_0s_0^2+144\g_0s_0t_0^2-27\g_0t_0^4+16\g_0^2s_0^4-4\g_0^2s_0^3t_0^2)\\
    &\cdot\{256u_1^3-\g_0(128s_1^2+144s_1t_1^2-27t_1^4+16\g_0s_1^4-4\g_0s_1^3t_1^2)\}.
\end{align*}
This shows that the first factor has an empty intersection with the exceptional divisor $(\g_0=0)$, whereas the second factor intersects the exceptional divisor non-transversally along $(u_1=0)$.

Next, we consider the action of the finite quotient $\SS_2\cong R/R^{\circ}$.
We only consider the intersection $(\a_0=u_0=0)$ on affine locus $\PP$ (the other cases being the same).
If $\diag(\lambda,\lambda^{-1})$ fixes a general point in $\PP\cap(\a_0=u_0=0)$, by the condition on $t_0$, we have $\lambda^2=1$.
This implies that $\diag(\lambda,\lambda^{-1})$ is trivial as an element of $\PGL_2(\C)$.

Thus, finally, let us consider the case of the form $\antidiag(\lambda,-\lambda^{-1})$.
This element swaps the coordinates $u_0$ and $u_1$.
However, a general point $p=(s_1,t_0,t_1,u_1)\in\PP\cap(\a_0=u_0=0)$ satisfies $u_1\neq 0$, and this implies $\antidiag(\lambda,-\lambda^{-1})$ cannot stabilize any point.
More explicitly, we have 
\[\antidiag(\lambda,-\lambda^{-1})\cdot(s_1,t_0,t_1,u_1)=(\lambda^{-16}s_1^{-1},\lambda^2t_1s_1^{-1}, -\lambda^{14}t_0s_1^{-1}, 0)\]
by (\ref{eq:action_anti_diag}).
This cannot be equal to $p=(s_1,t_0,t_1,u_1)$ with $u_1\neq 0$.
\end{proof}

\begin{rem}\label{rem:orderediso}
The situation in the ordered case is different. Indeed, a similar calculation, again using a  Luna slice argument, shows that the 
discriminant divisors and the boundary divisors meet transversally everywhere on $\M^{\K}_{\ord}$.
\end{rem}
\begin{rem}\label{rem:isoordered}
In Theorem \ref{thm:coh_ordered_tor}, we shall see that $\M_{\ord}^{\K}$ and $\overline{\B^5/\Gamma_{\ord}}^{\tor}$ have the same cohomology. Note that this proof does not require a priori knowledge that the two spaces are isomorphic.
Using the information of their Betti numbers, we can give a short independent proof that $\M_{\ord}^{\K} \cong \overline{\B^5/\Gamma_{\ord}}^{\tor}$ which is independent of \cite{GKS21}. This argument follows a similar argument given by Casalaina-Martin for cubic surfaces. By the Borel extension theorem \cite[Theorem A]{Bo72}, the map $\M_{\ord} \to \B^5/\Gamma_{\ord}$ extends to a morphism $\M_{\ord}^{\K} \to \overline{\B^5/\Gamma_{\ord}}^{\tor}$. Since both spaces have 
the same Betti numbers, this must be an isomorphism or a small contraction. But the latter is impossible since $\overline{\B^5/\Gamma_{\ord}}^{\tor}$ is $\Q$-factorial (and in fact smooth). 
\end{rem}

In the rest of this subsection, we work on the stabilizers of points in the exceptional divisor $\Delta$ in $\M^{\K}$.
The following proposition plays a critical role in the proof of Theorem \ref{thm:not_K_equiv}.

\begin{prop}
\label{prop:e_not_divisible_by_5}
For any point in $x\in\Delta$, the order of its stabilizer $S_x\defeq \Stab_R(x)$ is $\mathrm{not}$ divisible by 5.
\end{prop}
\begin{proof}
Since the order of the finite part of $R$ is not divisible by 5, it is enough to concentrate on the connected component $R^{\circ}$, which is isomorphic to $\C^{\times}$.
For simplicity, we will also use  $S_x$ to denote the stabilizer of $x$ in $R^{\circ}$.
By the $\SS_2$ symmetry, it suffices to show the claim for the affine open sets $\mathcal{P}$, $\mathcal{Q}$ and $\mathcal{R}$.

First, let us consider the points $(\a_0,s_1,t_0,t_1,u_0,u_1)\in\mathcal{P}$.
In this locus, the exceptional divisor corresponds to $(\a_0=0)$, and the action of $\diag(\lambda,\lambda^{-1})$ is given by 
\[\diag(\lambda,\lambda^{-1})\cdot(0,s_1,t_0,t_1,u_0,u_1)=(0,\lambda^{-16}s_1,\lambda^{-2}t_0,\lambda^{-14}t_1,\lambda^{-4}u_0,\lambda^{-12}u_1).\]
Since the Kirwan blow-up is completed after one step, it is enough to consider the stable points after blowing up the orbit $\SL_2(\C)\cdot \{c_{4,4}\}$.
It follows from Lemma \ref{lem:unstable_locus} that 
both \{$t_0\neq 0$ or $u_0\neq 0$\} and \{$s_1\neq 0$ or $t_1\neq 0$ or $u_1\neq 0$\}.
If $t_0\neq 0$, then $S_x\cong\Z/2\Z$.
If $u_0\neq 0$, then 
\[S_x\cong\begin{cases}
\Z/4\Z&(s_1\neq 0\ \mathrm{or}\ u_1\neq 0)\\
\Z/2\Z&(t_1\neq 0).
\end{cases}
\]

The other cases are similar, but we nevertheless state them for completeness, starting with the points $(s_0,s_1,0,t_1,u_0,u_1)\in\mathcal{P}\cap (\beta_0=0)$.
The action of $\diag(\lambda,\lambda^{-1})$ is given by 
\[\diag(\lambda,\lambda^{-1})\cdot(s_0,s_1,0,t_1,u_0,u_1)=(\lambda^2s_0,\lambda^{-14}s_1,0,\lambda^{-12}t_1,\lambda^{-2}u_0,\lambda^{-10}u_1).\]
Again by Lemma \ref{lem:unstable_locus}, we can assume that \{$s_0\neq 0$ or  $u_0\neq 0$\} and \{$s_1\neq 0$ or $t_1\neq 0$ or $u_1\neq 0$\}.
In all cases, we obtain $S_x\cong \Z/2\Z$.
Finally, let $(s_0,s_1,t_0,t_1,0,u_1)\in\mathcal{P}\cap (\gamma_0=0)$.
The action of $\diag(\lambda,\lambda^{-1})$ is given by 
\[\diag(\lambda,\lambda^{-1})\cdot(s_0,s_1,t_0,t_1,0,u_1)=(\lambda^4s_0,\lambda^{-4}s_1,\lambda^2t_0,\lambda^{-10}t_1,0,\lambda^{-8}u_1).\]
As above, we study the case holding both of \{$s_0\neq 0$ or $t_0\neq 0$\} and \{$s_1\neq 0$ or $t_1\neq 0$ or $u_1\neq 0$\}.
If $t_0\neq 0$, then $S_x\cong\Z/2\Z$.
If $s_0\neq 0$, then 
\[S_x\cong\begin{cases}
\Z/4\Z&(s_1\neq 0\ \mathrm{or}\ u_1\neq 0)\\
\Z/2\Z&(t_1\neq 0).
\end{cases}
\]
This calculation completes the proof.

\end{proof}

\subsection{Transversality in the toroidal compactification}
\label{transversality}
In this subsection, we prove that the discriminant divisors and the boundary divisors intersect transversally in $\overline{\B^5/\Gamma_{\ord}}^{\tor}$  and generically transversally in $\overline{\B/\Gamma}^{\tor}$.
We will see that this also implies the transversality at a generic point in $\overline{\B^5/\Gamma}^{\tor}$.
Throughout this subsection, let $\mathbb{N}_8\defeq\{1,2,\cdots,8\}$ and $I\subset\mathbb{N}_8$.
As before, $(\M_{\ord}^{\GIT})^o$ denotes the set of $8$-tuples where all points are different; see Section \ref{section:preparation}.
Below, we shall recall the construction of the blow-up sequence $\overline{\M}_{0,8}\to\M_{\ord}^{\K}\to\M_{\ord}^{\GIT}$.
By the explicit description of the blow-ups or the interpretation as the configuration space, the locus $(\M_{\ord}^{\GIT})^o$ does not meet the centres of each blow-up step.
Thus, we consider $(\M_{\ord}^{\GIT})^o$ to be also an open subset of $\overline{\M}_{0,8}$ and $\M_{\ord}^{\K}$ via birational morphisms.

First, we work on $\M^{\GIT}_{\ord}$.
The boundary divisor  $\M^{\GIT}_{\ord}\setminus (\M_{\ord}^{\GIT})^o$ is 
\[D^{(0)}_2\defeq\bigcup_{|I|=2}D_2^{(0)}(I)=\D_{\ord}\subset\M_{\ord}^{\GIT}\]
by \cite[p1134]{KM11} ($m=4, k=0$).
Here, $D_2^{(0)}(I)$ is defined by
\[D_2^{(0)}(I)\defeq\overline{\left\{(x_1,\cdots,x_8)\in (\P^1)^8\mid x_i=x_j\ \mathrm{if}\ i,j\in I  \right\}}//\SL_2(\C).\]
The number of such $I$ is 28.
As in Section \ref{section:preparation}, the morphism $\varphi_1:\M_{\ord}^{\K}\to\M_{\ord}^{\GIT}$ is the Kirwan blow-up whose centre is the locus of polystable orbits, consisting of 35 orbits (which in turn 
correspond to the $35$ cusps, see below).
We interpret $\varphi_1$ in terms of configuration spaces as follows.
Let 
\[\Sigma_4^{(0)}(I,I^{\perp})\defeq\overline{\left\{(x_1,\cdots,x_8)\in (\P^1)^8 \mid x_i=x_j\ \mathrm{if\ and\ only\ if}\ \{i,j\}\subset I\ \mathrm{or}\ \{i,j\}\subset I^{\perp} \right\}}//\SL_2(\C)\]
for $|I|=|I^{\perp}|=4$ and $I\sqcup I^{\perp}=\mathbb{N}_8$.
We also denote by $\Sigma_4^{(0)}$ their union running through such $I$ and $I^{\perp}$.
Note that there are 35 pairs $(I,I^{\perp})$ satisfying $|I|=|I^{\perp}|=4$ and $I\sqcup I^{\perp}=\mathbb{N}_8$.
In this terminology, the centre of $\varphi_1$ is described by
\[\Sigma_4^{(0)}=\{c_{\ord,i}\}_{i=1}^{35}\]
where $\{c_{\ord,i}\}_{i=1}^{35}$ are the polystable points of $\M^{\GIT}_{\ord}$, corresponding to 35 Baily-Borel cusps.

Next, we consider $\M^{\K}_{\ord}(\cong\overline{\M}_{0,8(\frac{1}{4}+\epsilon)})$.
Let 
\[D_4^{(1)}(I)\defeq\varphi_1^{-1}\left(\Sigma_4^{(0)}(I,I^{\perp})\right)\]
for $|I|=4$.
Then, the exceptional divisor of $\varphi_1$ is
\[D_4^{(1)}\defeq\bigcup_{|I|=4}D_4^{(1)}(I)=\varphi_1^{-1}\left(\Sigma_4^{(0)}\right)=\Delta_{\ord}.\]
Note that each irreducible component of $\Delta_{\ord}$ is isomorphic to $\P^2\times\P^2$ by \cite[Proposition 4.5]{Ha03}, \cite[Remark 6]{MS21} or \cite[Example 2.12]{GKS21}.
Besides, let 
\[D_2^{(1)}(I)\defeq\overline{\varphi_1^{-1}\left(D_2^{(0)}(I)\setminus\Sigma_4^{(0)}\right)}\]
be the strict transform of $D_2^{(0)}(I)$ for $|I|=2$, and $D_2^{(1)}$ be their union.
Then $D_2^{(1)}$ is the strict transform of $\D_{\ord}$, i.e.,
\[D_2^{(1)}=\widetilde{\D_{\ord}}\]
and has 28 irreducible components.
In this setting, the boundary divisor $\M^{\K}_{\ord}\setminus (\M_{\ord}^{\GIT})^o$ is 
\[D^{(1)}_2\bigcup D^{(1)}_4=\widetilde{\D_{\ord}}\bigcup \Delta_{\ord}\]
by \cite[p1134]{KM11} ($m=4, k=1$).

Next, we describe the centre of the blow-up $\varphi_2\defeq\varphi_2'\circ\varphi_{\ord}:\overline{\M}_{0,8}\to\M^{\K}_{\ord}$, 
which is a codimension 2 locus.
Let 
\begin{align*}
\Sigma_3^{(0)}(I)&\defeq\overline{\left\{(x_1,\cdots,x_8)\in (\P^1)^8 \mid x_i=x_j\ \mathrm{if\ and\ only\ if}\ i,j\in I \right\}}//\SL_2(\C)\\
\Sigma_3^{(1)}(I)&\defeq\overline{\varphi_1^{-1}\left(\Sigma_3^{(0)}(I)\setminus\Sigma_4^{(0)}\right)}
\end{align*}
for $|I|=3$ and
\[\Sigma_3^{(1)}\defeq\bigcup_{|I|=3}\Sigma_3^{(1)}(I).\]
Then, the centre of the blow up $\varphi_2:\overline{\M_{0,8}}\to\M_{\ord}^{\K}$ 
is $\Sigma_3^{(1)}$.

Finally, we study $\overline{\M}_{0,8}$.
For $|I|=3$, let 
\[D_3^{(2)}(I)\defeq\varphi_2^{-1}\left(\Sigma_3^{(1)}(I)\right)\]
be an irreducible component of the exceptional divisor of $\varphi_2$.
Then, the variety
\[D_3^{(2)}\defeq\bigcup_{|I|=3}D_3^{(2)}(I)=\varphi_2^{-1}\left(\Sigma_3^{(1)}\right)\]
 is exactly the exceptional divisor of the blow-up $\varphi_2$.
 For $|I|=2,4$, we denote by $D_{|I|}^{(2)}(I)$ the strict transform of $D_{|I|}^{(1)}$ and define 
\[D_2^{(2)}\defeq\bigcup_{|I|=2}D_2^{(2)}(I),\quad D_4^{(2)}\defeq\bigcup_{|I|=4}D_4^{(2)}(I).\]
Now, the boundary divisor  $\overline{\M}_{0,8}\setminus (\M_{\ord}^{\GIT})^o$ 
is
\[D^{(2)}_2\bigcup D^{(2)}_3 \bigcup D^{(2)}_4\]
by \cite[p1134]{KM11} ($m=4, k=2$).

The boundaries which are contracted through the map $\varphi_2$ can also be calculated as follows.
By \cite[Theorem 4.1]{Ha03}, there exists the reduction map
\[\varphi_2':\overline{\M}_{0,8}\to\overline{\M}_{0,8(\frac{1}{4}+\epsilon)}.\]
The map $\varphi_2'$ is a divisorial contraction, more precisely:

\begin{lem}[c.f. {\cite[Proposition 4.5]{Ha03}}]
The morphism $\varphi_2'$ contracts the boundary divisors $D_3^{(2)}$.
\end{lem}
\begin{proof}
By \cite[p1121]{KM11}, the exceptional locus of $\varphi_2'$ is the union of $D_{|I|}^{(2)}(I)$ with $I=\{i_1,\dots,i_r\}$ for $r>2$ so that 
\[r\times (\frac{1}{4}+\epsilon)\leq 1.\]
This implies $r=3$.
\end{proof}

By construction, $D^{(2)}_2\cup D^{(2)}_3 \cup D^{(2)}_4$ is normal crossing (since $\overline{M}_{0.8}$ is a normal crossing compactification of $(\B^5\setminus H)/\Gamma_{\ord}$).
We denote $\H_{\ord}\defeq\overline{H/\Gamma_{\ord}}$ and $\H\defeq\overline{H/\Gamma}$, where the closures are taken in the respective Baily-Borel compactifications.
We further denote by $\widetilde{\H_{\ord}}$ the strict transform of $\H_{\ord}$ under $\pi_{\ord}:\overline{\B^5/\Gamma_{\ord}}^{\tor}\to\overline{\B^5/\Gamma_{\ord}}^{\BB}$.
 Since the contraction divisor of $\varphi_2$ is only $D^{(2)}_3$, we now obtain the following:

\begin{thm}
\label{thm:normal_crossing_ord}
The boundary $\widetilde{\H_{\ord}}\cup T_{\ord}$ is a normal crossing divisor.
In particular, $\widetilde{\H_{\ord}}$ and $T_{\ord}$ intersect transversally everywhere in $\overline{\B^5/\Gamma_{\ord}}^{\tor}$.
\end{thm}
Again, by this formulation, we mean that $\widetilde{\H_{\ord}}$ and $T_{\ord}$ intersect  transversally everywhere along any component of their intersection. 
As a consequence, we obtain the following corollary, 
where $\widetilde{\H}$ is the strict transform of $\H$ under $\pi:\overline{\B^5/\Gamma}^{\tor}\to\overline{\B^5/\Gamma}^{\BB}$.
\begin{cor}
The divisor $\widetilde{\H}\cup T$ is a  normal crossing divisor, up to finite quotients.
\end{cor}

Next, we discuss the generical transversality of the intersection of $\widetilde{\H}$ and $T$ in $\overline{\B^5/\Gamma}^{\tor}$. 
Note that $\Gamma/\Gamma_{\ord}\cong \SS_8$ acts on $\{T_{\ord,i}\}_{i=1}^{35}$ transitively and
\[1\to \SS_4\times \SS_4\to \Stab_{\SS_8}(T_{\ord,i})\to \SS_2\to 1.\]
Next, we study the description of the boundary and group actions via the Hermitian form.
The claim of the following lemma is already known in terms of a moduli description by \cite[Remark 6]{MS21} or \cite[Example 2.12]{GKS21}, 
but we need the details in the proof of Theorem  \ref{thm:nonord_transversal}.
\begin{lem}
\label{lem:boundary}
The following holds.
\begin{enumerate}
    \item $T_{\ord,i}\cong\P^2\times\P^2$.
    \item $T\cong\left(\P^2/\SS_4\times\P^2/\SS_4\right)/\SS_2$.
\end{enumerate}
\end{lem}
\begin{proof}
    We orientate ourselves along the strategy of the proof of \cite[Proposition 7.8]{CMGHL23}.
    First, we take an isotropic vector $h=(1,0,0,0,0,0)\in L$ and denote by $F$ the corresponding cusp.
    As the unitary group acts transitively on the set of all cusps, this means no loss of generality.
    Also, taking $h^{\vee}=(0,1,0,0,0,0)$ as a further basis vector, we can replace our  Hermitian form by 
\[\begin{pmatrix}
 &  & 1-\sqrt{-1} \\
  & B & \\
1+\sqrt{-1}  &  &  \\
\end{pmatrix}\]
where
\[B\defeq\begin{pmatrix}
-2 & 1+\sqrt{-1} &  & \\
1-\sqrt{-1}  & -2 & & \\
  &  & -2 & 1+\sqrt{-1} \\
  &  & 1-\sqrt{-1} & -2  \\
\end{pmatrix}.\]
Then, 
\[N(F)\defeq\Stab_{\Gamma}(F)=\left\{g=\begin{pmatrix}
u & v & w  \\
  & X & y \\
  &  & s \\
\end{pmatrix}\ \middle|\ \begin{aligned}
&s\overline{u}=1,\ \overline{X^t}BX=B \\ &\overline{X^t}By+(1-\sqrt{-1})\overline{v^t}s=0\\ 
&\overline{y^t}By+(1+\sqrt{-1})\overline{s}w+(1-\sqrt{-1})s\overline{w}=0
\end{aligned}\right\}.\]
The unipotent radical of $N(F)$ is
\[W(F)=\left\{g=\begin{pmatrix}
1 & v & w  \\
  & I_{4} & y \\
  &  & 1 \\
\end{pmatrix}\ \middle|\ \begin{aligned}
&By+(1-\sqrt{-1})\overline{v^t}=0\\ 
&\overline{y^t}By+(1+\sqrt{-1})w+(1-\sqrt{-1})\overline{w}=0
\end{aligned}\right\}\]
    and the centre of $W(F)$ is 
    \[Z(F)=\left\{g=\begin{pmatrix}
1 &  & \sqrt{-1}(1-\sqrt{-1})w  \\
  & I_{4} &  \\
  &  & 1 \\
\end{pmatrix}\ \middle|\ \begin{aligned}
w\in\Z
\end{aligned}\right\}.\]
We take the partial quotient of $\B^5$ by the action of $Z(F)$:
\[\begin{array}{ccc}
\B^5                     &\hookrightarrow& \C^{\times}\times\C^4                  \\
        (z_0,z_1,z_2,z_3,z_4)                    & \mapsto   & (t=\exp\left(2\pi z_0/(1-\sqrt{-1})\right),z_1,z_2,z_3,z_4).
\end{array}\]
We shall here consider the quotient of $\C^4$ by $W(F)$.
For an element $g\in W(F)$, its action on $\underline{z}\defeq(z_1,z_2,z_3,z_4)$ is given by
\[g\cdot \underline{z}^t=\frac{1}{s}(X\underline{z}+y).\]
A straightforward computation shows that for given $y^t\in\Z[\sqrt{-1}]^4$, we can find suitable elements $w\in\Z[\sqrt{-1}]$ and $v\in\Z[\sqrt{-1}]^4$
such that $g=\begin{pmatrix}
1 & v & w  \\
  & I_4 & y \\
  &  & 1 \\
\end{pmatrix} \in W(F)$.
This implies that
\[\C^4/W(F)\cong (E_{\sqrt{-1}})^4,\]
where $E_{\sqrt{-1}}$ is the CM-elliptic curve $\C/\left(\Z+\sqrt{-1}\Z\right)$.
Now, we consider the effect of an element of the form  
\[g=\begin{pmatrix}
u &  &   \\
  & I_4 &  \\
  &  & s \\
\end{pmatrix}\in N(F).\]
Here, from the above action, $s\in\Z[\sqrt{-1}]^{\times}$ acts on $(E_{\sqrt{-1}})^4$ diagonally by multiplication with powers of $\sqrt{-1}$.
However, this element is already in $\U(D_4^{\oplus 2})$, thus it follows that 
$T\cong(E_{\sqrt{-1}})^4/\U(D_4^{\oplus 2})$.
Here, we note that $X=\U(D_4^{\oplus 2})$.
By \cite[Table 2]{Do08}, we have
\[\U(D_4^{\oplus 2})\cong (\left(\Z/2\Z\right)^2\times\SS_2)\rtimes\SS_4)^2\rtimes\SS_2.\]
See also \cite[Subsection 6.4]{Sh53}.
Since, the action of this group, described in \cite[Subsection 3.2, Table 2]{Do08}, gives 
\[(E_{\sqrt{-1}})^2/\U(D_4)\cong (\P^1)^2/\left(\SS_2\rtimes\SS_4\right) \cong \P^2/\SS_4,\]
where $\SS_4$ acts on $\P^2$ by the standard representation, 
we obtain 
\[(E_{\sqrt{-1}})^4/\U(D_4^{\oplus 2})\cong (\P^2/\SS_4)^2/\SS_2. \]
For the ordered case, a straightforward computation shows that $\widetilde{\U}(D_4)\cong(\Z/2\Z)^2\times\SS_2$, thus this gives 
\[T_{\ord,i}\cong\P^2\times\P^2.\]
\end{proof}

\begin{rem}\label{rem:structureT}
This description allows us to describe the geometry of the toroidal boundary $T$ explicitly. By the above Lemma \ref{lem:boundary} we know that
$T=(\P^2/\SS_4 \times \P^2/\SS_4)/\SS_2$ where $\SS_4$ acts on $\P^2$ by the standard $3$-dimensional representation and $\SS_2$ exchanges 
the two factors. 
We claim that $\P^2/\SS_4 \cong \P(1,2,3)$ where $\P(1,2,3)$ denotes the weighted projective space with weights $(1,2,3)$. 
This follows since the invariants are freely generated by the restriction of the elementary symmetric polynomials of degree $2,3,4$ 
on $\P^3$ restricted to the hyperplane $\sum_{i=0}^3 x_i=0$. Hence $\P^2/\SS_4 \cong \P(2,3,4) \cong \P(1,2,3)$. In conclusion we find that $T \cong S^2(\P(1,2,3))$.
\end{rem}

Before discussing the intersection of divisors on the toroidal compactifications, we recall the discriminant form, see \cite[Subsection 2.3]{Ko07a} (where the lattice is called $N$ compared to our $L$):
\[q_L:A_L\to \mathbb{F}_2.\]
Associated with $q_L$, there is an associated bilinear form $b_L(\ ,\ )$ on $A_L$.
Note that $q_L$ is isomorphic to the direct sum of 3 copies of the hyperbolic plane $u$
over $\mathbb{F}_2$ which follows from \cite[Subsection 2.2]{Ko07a} or explicit computation in terms of the concrete form of $L$.
Here, $u$ is defined by the matrix
\[\begin{pmatrix}
    0&1\\
    1&0\\
\end{pmatrix}\]
over $\mathbb{F}_2$.
We have to pay attention to the norm of a vector because our quadratic form exists over $\mathbb{F}_2$.
In other words, the norm is measured by $q_L$, not $b_L(\ ,\ )$.
\begin{lem}
\label{lem:finite_quadratic_form}
For a given isotropic vector $h$ in the finite quadratic space $\P(A_L)\cong \P(\mathbb{F}_2^6)$, the orthogonal complement $h^{\perp}\cong\P(\mathbb{F}_2^5)$ contains 19 isotropic vectors and 12 non-isotropic vectors.
In addition, the stabilizer of $\Stab(h)$ in $\SS_8$ acts transitively on the set consisting of all 12 non-isotropic vectors.
\end{lem}
\begin{proof}
By \cite[Section 3]{FM11} or \cite[Proposition 3.2]{MT04}, we have an isomorphism $\Gamma/\Gamma_{\ord}\cong\SS_8\cong\O(\mathbb{F}_2^6)$. This naturally induces the action of the symmetric group $\SS_8$ on the discriminant group.
By \cite[Proposition 4.7 (ii)]{Ko07a}, the action on the set of isotropic vectors is transitive. 
Hence it suffices to consider one isotropic vector, say $h=(1,0,0,0,0,0)\in u^{\oplus 3}$.
Then, the non-isotropic vectors in $h^{\perp}$ are given by $(0,0,1,1,0,0), (0,0,1,1,1,0)$, $(0,0,1,1,0,1)$, $(0,1,1,1,0,0)$, $(0,1,1,1,1,0)$, $(0,1,1,1,0,1)$ and the vectors which arise from these by interchanging the last two components of $u^{\oplus 3}$.
Similarly, one obtains a complete list of isotropic vectors in $h^{\perp}$ (which contains $h$ itself).
The latter half of the statement is clear because for any two non-isotropic vectors $v_1$ and $v_2$, orthogonal to $h$, we can define an element $g\in\Stab(h)$ permuting $v_1$ and $v_2$, and extend it by the identity to $\l v_1,v_2,h\r^{\perp}\subset\mathbb{F}_2^6$.
Here, we used the fact that there is no relation such as $h=v_1+v_2$, i.e., that $v_1, v_2$ and $h$ are independent.
\end{proof}

The goal of this subsection is the following theorem.
\begin{thm}
\label{thm:nonord_transversal}
The divisors $\widetilde{\H}$ and $T$ meet generically transversally in $\overline{\B/\Gamma}^{\tor}$.
\end{thm}
\begin{proof}
First, we take an irreducible component $T_{\ord,i}$ of $T_{\ord}$, namely the divisor over the cusp corresponding to the isotropic vector $h=(1,1,0,0,0,0)$.  
Then, we choose the component of $\widetilde{\H_{\ord}}\cap T_{\ord,i}$ given by taking the divisor orthogonal to the vector $\ell=(0,0,1,1,0,0)$. We can perform both choices without loss of generality due to 
Lemma \ref{lem:finite_quadratic_form}, which tells us that the group $\SS_8$ acts transitively on the components of $\widetilde{\H_{\ord}}\cap T_{\ord,i}$.

Thus, it suffices to consider the component $\mathcal{T}$ of $\widetilde{\H_{\ord}}\cap T_{\ord,i}$ chosen above. 
Now, $\mathcal{T}$ is the fixed locus of the reflection with respect to $\ell$.
In addition, through the isomorphism $\Gamma/\Gamma_{\ord}\cong\SS_8\cong\O(\mathbb{F}_2^6)$, the choice of $\ell$ implies that this reflection acts on $\P^2\times\P^2$ by \begin{align*}
\P^2\times\P^2&\to\P^2\times\P^2\\
    ([a_1:b_1:c_1],[a_2:b_2:c_2])&\mapsto ([b_1:a_1:c_1],[a_2:b_2:c_2]).
\end{align*}
Also, a straightforward computation shows that  $\mathcal{T}$ is not fixed by any other reflection with respect to a non-isotropic vector set-theoretically. 
Hence, we consider a general point $p=(p_1,p_2)\in\mathcal{T}\subset\P^2\times\P^2$, where general means the following: the point $p_1=[1:1:c] \in \P^2$ satisfies $\Stab_{\SS_4}(p_1)=\l(1\ 2)\r$, where $(1\ 2)$ denotes the transposition in $\SS_4$ of the first two components, 
and $p_2$ is general in the sense that $p_1\neq p_2$ and $\Stab_{\SS_4}(p_2)=1$.
Clearly, the set of these points is non-empty.
Here, we have used the fact that $\SS_4$ acts on $\P^2$ by the standard representation; see the proof of Lemma  \ref{lem:boundary} and  \cite[Subsection 3.2]{Do08}.
By construction, the stabilizer of $p$ is isomorphic to $\Z/2\Z$, generated by a non-trivial involution in the first factor of $\SS_4\times\SS_4$.

Using the coordinates taken in the proof of Lemma \ref{lem:boundary}, 
by Theorem \ref{thm:normal_crossing_ord}, taking the quotients, we can choose the defining equation of $T_{\ord}$ (resp. $\widetilde{\H_{\ord}}$) as $(t=0)$ (resp. $(z_1=0)$).
Then, the non-trivial involution in $\Stab(p)$ acts on $p$ as $(t,z_1,z_2,z_3,z_4)\mapsto (t,-z_1,z_2,z_3,z_4)$.
Hence, we obtain the new coordinates $(t, w_1, z_2, z_3, z_4)$ of $\overline{\B^5/\Gamma}^{\tor}$, where $w_1=z_1^2$.
Therefore, the divisors $T$ and $\widetilde{\H}$, defined by $(t=0)$ and $(w_1=0)$ respectively, meet transversally.
\end{proof}

\subsection{Proof of Theorem \ref{mainthm:extendability}}
We shall now restate one of the main results in this paper. Its proof uses our computation of the Betti numbers of the Kirwan blow-up $\M^{\K}$ and the  toroidal compactification $\overline{\B^5/\Gamma}^{\tor}$
which we will perform in Section \ref{sec:Bettinumbers}.
\begin{thm}
\label{thm:not_extend}
Neither the Deligne-Mostow isomorphism $\phi:\M^{\GIT}\to\overline{\B^5/\Gamma}^{\BB}$ nor its inverse $\phi^{-1}$ lift to a morphism between the Kirwan blow-up $\M^{\K}$ and the unique toroidal compactification 
$\overline{\B^5/\Gamma}^{\tor}$.
\end{thm}
\begin{proof}
We shall prove this for $\phi$, the argument for $\phi^{-1}$ being the same. 
By Theorem \ref{thm:nonord_nontransversal} and Theorem  \ref{thm:nonord_transversal}, the birational map $g:\M^{\K}\dashrightarrow \overline{\B^5/\Gamma}^{\tor}$ cannot be an isomorphism.
By Theorems \ref{thm:coh_M^K} and \ref{thm:coh_tor} the Betti numbers $b_2(\M^{\K})=b_2( \overline{\B^5/\Gamma}^{\tor})=2$ agree. 
Hence $g$ cannot contract a divisor and must thus be a small contraction.
This, however, contradicts the fact that both $\M^{\K}$ and  $\overline{\B^5/\Gamma}^{\tor}$ are $\Q$-factorial.
(See also the proof of \cite[Theorem 1.1]{CMGHL24}).
\end{proof}
Since the compactifications concerned are $\SS_8$-equivariant, we obtain as a byproduct that $\M_{\ord}^{\K}/\SS_8\not\cong \M^{\K}$.

\section{Canonical bundles and relation to the minimal model program}
\label{section:canonical_bundles}
On the way, we shall use a modular form constructed by Kond\={o}, which will be essential for us.
In this section, we focus on the canonical bundles, and as a result, we shall show Theorem \ref{mainthm:not_K_equiv}.
\subsection{Computation involving blow-ups}
We first recall some basic facts about the birational geometry of the relevant moduli spaces and, noticeably, the maps $\varphi_1$ and $\varphi_2$.
These two morphisms  $\varphi_1:\M_{\ord}^{\K}\to\M_{\ord}^{\GIT}$ and $\varphi_2:\overline{\M}_{0,8}\to\M^{\K}_{\ord}$ are the blow-ups given by the reduction of the weights \cite[Theorem 4.1]{Ha03}; see Figure \ref{com_diag} and Subsection \ref{transversality}.
From \cite[Lemma 5.3]{KM11}, in their $\Q$-Picard groups, we obtain 
\begin{align}
\label{eq:pull_back}
    \varphi_1^{*}(D_2^{(0)})&=D_2^{(1)}+6D_4^{(1)}\\
    &=\widetilde{\D_{\ord}}+6\Delta_{\ord}\notag
\end{align}
where we simply relabel the divisors in the second equality. Using the isomorphism $\M_{\ord}^{\K} \cong \B^5/\Gamma_{\ord}$ this, of course, implies 
\begin{align}
\label{eq:pull_back_toroidal}
    \pi_{\ord}^{*}(\H_{\ord})=\widetilde{\H_{\ord}}+6T_{\ord}.
\end{align}
We  also note that \cite[Lemma 5.3]{KM11} implies that
\begin{equation}
\label{equ:pullbackphi2D21}
\varphi_2^{*}(D_2^{(1)})=D_2^{(2)}+3D_3^{(2)},\quad \varphi_2^{*}(D_4^{(1)})=D_4^{(2).}
\end{equation}
For the sake of completeness, we also remark that
\[\varphi_{2*}(D_2^{(2)})=D_2^{(1)}=\widetilde{\D_{\ord}},\quad \varphi_{2*}(D_3^{(2)})=0,\quad \varphi_{2*}(D_4^{(2)})=D_4^{(1)}=T_{\ord},\]
\[\varphi_{1*}(D_2^{(1)})=D_2^{(0)}=\D_{\ord},\quad  \varphi_{1*}(D_4^{(1)})=0,\]
The pushforward formulae are not used in this paper.
All of these equalities hold in the relevant $\Q$-Picard groups.

Moreover, the canonical divisors are described as 
\begin{align}\label{eq:can_bindle_M^K_{ord}}
 K_{\overline{\M}_{0,8}}&=-\frac{2}{7}D^{(2)}_2+\frac{1}{7}D^{(2)}_3+\frac{2}{7}D^{(2)}_4\notag\\
  K_{\M_{\ord}^{\K}}&=-\frac{2}{7}D^{(1)}_2+\frac{2}{7}D^{(1)}_4\\
    K_{\M_{\ord}^{\GIT}}&=-\frac{2}{7}D^{(0)}_2\notag
\end{align}
where the number $7$ in the denominators comes from $n-1$ in \cite[Proposition 5.4, Lemma 5.5]{KM11}.
Combining (\ref{eq:pull_back}), (\ref{equ:pullbackphi2D21}) and (\ref{eq:can_bindle_M^K_{ord}}) it follows that 
\begin{align*}
    K_{\M^{\K}_{\ord}}&=\varphi_1^{*}(K_{\M^{\GIT}_{\ord}})+2D_4^{(1)}\\
     K_{\overline{\M}_{0,8}}&=\varphi_2^{*}(K_{\M^{\K}_{\ord}})+D_3^{(2)}.
\end{align*}

In addition,  there is a modular form $F$ of weight 14 on $\B^5$ vanishing exactly on $H$ \cite[Theorem 6.2]{Ko07a} with vanishing order 1.
It follows that $\div(F) = H$ on $\B^5$.
Since the quotient map $\B^5\to\B^5/\Gamma_{\ord}$ ramifies along only $H$ with index 2, the construction of $\L_{\ord}$ by Baily and Borel implies that the above relation decends to 
\begin{equation}\label{eq:Kondoform}
14\L_{\ord}=\frac{1}{2}\H_{\ord}
\end{equation}
in $\Pic(\overline{\B^5/\Gamma_{\ord}}^{\BB})\otimes\Q$.
Note that the factor $1/2$ comes from the ramification index along $\H_{\ord}$ arising from the action of the arithmetic subgroup $\Gamma_{\ord}$.
Here $\L_{\ord}$ denotes the automorphic line bundle of weight 1.
By (standard) abuse of notation, we use the same notation for this line bundle on both the Baily-Borel and toroidal compactifications.
Thus, 
\begin{align*}
    K_{\overline{\B^5/\Gamma_{\ord}}^{\BB}}&=-\frac{2}{7}\D_{\ord}\\
    &=-8\L_{\ord}.
\end{align*}

Now, we compute the canonical bundles of $\overline{\B^5/\Gamma_{\ord}}^{\tor}\cong \M^{\K}_{\ord}$ in two ways: the realization as a ball quotient and the blow-up sequence.
\begin{rem}
\label{rem:ram}
The finite map $\B^5\to\B^5/\Gamma_{\ord}$ (resp. $\B^5/\Gamma_{\ord}\to\B^5/\Gamma$) branches along $H/\Gamma_{\ord}$ (resp. $H/\Gamma$) with branch index 2
Here we give a sketch of the proof. 
First, for $r\in L$ let 
\[\sigma_{\ell,\zeta}(r)\defeq r+(1-\zeta)\frac{\l\ell, r\r}{2}{\ell}\in L\otimes\Q(\sqrt{-1})\]
where $\ell\in L$ is a $(-2)$-vector and  $\zeta\in\{-1,\sqrt{-1}\}$. 
By \cite[Corollary 3 (ii)]{Be12} every quasi-reflection comes from such an automorphism. 
A straightforward calculation shows $\sigma_{\ell,-1}\in\Gamma_{\ord}$ and $\sigma_{\ell,\sqrt{-1}}\in\Gamma\setminus\Gamma_{\ord}$.
This implies that the ramification by the finite group $\Gamma/\Gamma_{\ord}$ is induced by an order 2 element $\sigma_{\ell,\sqrt{-1}}$, and hence the ramification index of the map
$\B^5\to\B^5/\Gamma_{\ord}$ is 2. The claim for $\B^5/\Gamma_{\ord}\to\B^5/\Gamma$ also follows from this consideration.
\end{rem}
On the one hand, 
by Remark \ref{rem:ram}, a standard application of Hirzebruch's proportionality principle \cite{Mu77} gives
\begin{alignat*}{2}
    K_{\overline{\B^5/\Gamma_{\ord}}^{\tor}}&=6\L_{\ord}-\frac{1}{2}\widetilde{\H_{\ord}}-T_{\ord}\\
        &=6\L_{\ord}-\frac{1}{2}\left\{\pi_{\ord}^{*}(\H_{\ord})-6T_{\ord}\right\}-T_{\ord}&\quad (\mathrm{by}\ (\ref{eq:pull_back_toroidal}))\\
    &=-8\L_{\ord}+2T_{\ord}&\quad (\mathrm{by}\ (\ref{eq:Kondoform})).
\end{alignat*}

On the other hand, 
\begin{alignat*}{2}
   K_{\M^{\K}_{\ord}}
    &=-\frac{2}{7}\widetilde{\D_{\ord}}+\frac{2}{7}\Delta_{\ord}\quad&(\mathrm{by}\ (\ref{eq:can_bindle_M^K_{ord}}))\\
    &=-\frac{2}{7}\left\{\varphi_1^{*}(\D_{\ord})-6\Delta_{\ord}\right\}+\frac{2}{7}\Delta_{\ord}\quad&(\mathrm{by}\ (\ref{eq:pull_back}))\\
    &=-\frac{2}{7}\varphi_1^{*}\phi_{\ord}^{*}(\H_{\ord})+2\Delta_{\ord}\\
    &=-8\varphi_1^{*}\phi_{\ord}^{*}(\L_{\ord})+2\Delta_{\ord}\quad& (\mathrm{by}\ (\ref{eq:Kondoform}))\\
    &=\tau^{*}(-8\L_{\ord}+2T_{\ord})\quad &(\mathrm{by\ Figure}\ \ref{com_diag}),
\end{alignat*}
for $\tau\defeq\Phi_{\frac{1}{4}+\epsilon}\circ\phi_{\ord}$.
Thus, this calculation recovers the fact $K_{\M^{\K}_{\ord}}=\tau^{*}(K_{\overline{\B^5/\Gamma_{\ord}}^{\tor}})$ under the isomorphism $\tau:\M^{\K}_{\ord}\cong\overline{\B^5/\Gamma_{\ord}}^{\tor}$.

\begin{rem}
The above modular form constructed by Kond\={o} is a ``special reflective modular form" in the sense of \cite[Assumption 2.1]{MO23}.
Hence, both $\M_{\ord}^{\GIT}$ and $\M^{\GIT}$ are Fano varieties from the above computation or \cite[Theorem 2.4]{MO23}.
\end{rem}

Now, we need the description of normal bundles along the toroidal boundary. For this we recall from Lemma \ref{lem:boundary} that $T_{\ord, i}{\cong \P^2 \times \P^2}$.
\begin{prop}
\label{prop:normal_bundles}
The normal bundle of $T_{\ord, i}{\cong \P^2 \times \P^2}$ in $\overline{\B^5/\Gamma}^{\tor}$ is given by
\[\mathcal{N}_{T_{\ord, i}/\overline{\B^5/\Gamma}^{\tor}}=\OO(-1,-1).\]
\end{prop}
\begin{proof}
First, we obtain
\[(K_{\overline{\B^5/\Gamma_{\ord}}^{\tor}}+T_{\ord,i})\vert_{T_{\ord,i}}=(-8\L_{\ord}+2T_{\ord}+T_{\ord,i})\vert_{T_{\ord,i}}.\]
The left-hand side gives
\begin{align*}
    (K_{\overline{\B^5/\Gamma_{\ord}}^{\tor}}+T_{\ord,i})\vert_{T_{\ord,i}}&=K_{T_{\ord,i}}\\
    &=\OO(-3,-3)
\end{align*}
by the adjunction formula.
On the other hand, the right-hand side is
\begin{align*}
    (-8\L_{\ord}+2T_{\ord}+T_{\ord,i})\vert_{T_{\ord,i}}
    &= 3T_{\ord,i}\vert_{T_{\ord,i}}\\
    &=3\mathcal{N}_{T_{\ord, i}/\overline{\B^5/\Gamma}^{\tor}}.
\end{align*}

Here we use that $\L_{\ord}\vert_{T_{\ord},i} = 0$. For this, we recall that the automorphic $\Q$-line bundle $\L_{\ord}$ on $\overline{\B^5/\Gamma_{\ord}}^{\tor}$ is defined as 
the pullback of $\L_{\ord}$ on $\overline{\B^5/\Gamma_{\ord}}^{\BB}$ (by a standard abuse of notation both line bundles are denoted by the same symbol)  and hence is trivial along the exceptional divisors for the blow-up $\pi_{\ord}:\overline{\B^5/\Gamma_{\ord}}^{\tor}\to\overline{\B^5/\Gamma_{\ord}}^{\BB}$.
This completes the proof.
\end{proof}

\begin{rem}
\label{rem:cross_ratio}
This is an analogue of Naruki's result \cite[Proposition 12.1]{Na82} on the moduli spaces of cubic surfaces.
He constructed a cross ratio variety and analysed its singularity at the boundary.
Later, Gallardo-Kerr-Schaffler \cite[Theorem 1.4]{GKS21} showed that the toroidal compactification and Naruki's compactification are isomorphic and Casalaina-Martin-Grushevsky-Hulek-Laza \cite[Theorem 1.2]{CMGHL24} used this to compute the top self-intersection number of the canonical bundles.
In the case of the moduli spaces of 8 points, there also exists the cross ratio variety constructed by \cite[Theorem 2.4]{FM11}, \cite[Theorem 7.2]{Ko07a} or \cite[Theorem 1.1]{MT04}.
However, these coincide with the Baily-Borel compactification $\overline{\B^5/\Gamma_{\ord}}^{\BB}$ of the ball quotient unlike the case of cubic surfaces.
This is why we used the results on the moduli spaces of stable curves in our case.
\end{rem}

Now, we study the behaviour of the boundary divisors along the finite covering $\overline{\B^5/\Gamma_{\ord}}^{\tor}\to\overline{\B^5/\Gamma}^{\tor}$.
We recall that the toroidal compactifications are constructed by taking a  ``partial compactification in the direction of each cusp" \cite[Section III. 5]{AMRT10}.
Here, this is done by choosing a polyhedral decomposition of a cone in the centre of the unipotent part of the stabilizer of a cusp (which is canonical in our case).
Hence, this group, which is denoted by $U(F)$ in \cite{AMRT10}, describes the toroidal boundary.
\begin{lem}
\label{lem:boundary_branch}
    The map  $\overline{\B^5/\Gamma_{\ord}}^{\tor}\to\overline{\B^5/\Gamma}^{\tor}$ does not branch along $T$.
\end{lem}
\begin{proof}
We recall from Lemma \ref{lem:finite_quadratic_form} and its proof that
the quotient $\Gamma/\Gamma_{\ord}\cong\SS_8$ acts transitively on the set of toroidal boundary components $\left\{T_{\ord,i}\right\}_{i=1}^{35}$ , since its action on the set of isotropic vectors in $A_L$ is transitive.
Hence, it suffices to take one component $T_{\ord,i}$, corresponding to the following isotropic vector $h\in L$, and prove that the centre, denoted as $Z(F)$ in Lemma \ref{lem:boundary}, of the unipotent radical of $\Stab_{\Gamma}(h)$ and $\Stab_{\Gamma_{\ord}}(h)$ are equal.
Now, we choose an isotropic vector $h\defeq(1,0,0,0,0)\in U\oplus U(2)\oplus D_4(-1)^{\oplus 2}$.
Then, the corresponding centre of the unipotent part of $\Stab_{\Gamma}(h)$ is given by
\[\left\{\left(
\begin{array}{c|c|c}
1 &   &  \sqrt{-1}(1-\sqrt{-1})w\\\hline
 & I_{4}  &  \\\hline
 &  & 1
\end{array}
\right)\ \middle|\ w\in\Z\right\}.
\]
Then, one can check that each matrix of the above form acts trivially on the discriminant group (see Section \ref{section:preparation}) $A_L= L^{\vee}/L$, which is isomorphic to $\left(\OO_F/(1+\sqrt{-1})\OO_F\right)^6$.
In other words, the unipotent radicals of $\Stab_{\Gamma}(h)$ and $\Stab_{\Gamma_{\ord}}(h)$ are equal.
We finally remark that there are  
 no irregular cusps in the sense of \cite{Ma23}. These can only occur when the arithmetic group in question contains an element of the center in the full unitary group.    
In this case, however, the discriminant group is isomorphic to $\Z/2\Z$ (as $\Z$-module), and $-\id$ and $-\sqrt{-1}\id$ act trivially on the discriminant and are thus already contained in $\Gamma_{\ord}$.
Altogether, this proves the claim.
\end{proof}
On the one hand, in a similar way as \cite[Proposition 5.8]{CMGHL24}, it follows that
    \begin{align}
\label{nonord_tor_K}
K_{\overline{\B^5/\Gamma}^{\tor}}=\pi^{*}K_{\overline{\B^5/\Gamma}^{\BB}}+7T
\end{align}
by Lemma \ref{lem:boundary_branch}.
On the other hand, we can calculate the canonical bundle of $\M^{\K}$ by \cite[Lemma 6.4]{CMGHL24}, where a general approach to calculating the canonical bundle of Kirwan blow-ups was developed:
\begin{align}
\label{nonord_Kirwan_K}
    K_{\M^{\K}}=f^{*}K_{\M^{\GIT}}+5\Delta,
\end{align}
where $\Delta$ is the exceptional divisor of the blow-up $f:\M^{\K}\to\M^{\GIT}$.
Here, we apply the method \cite[Lemma 6.4]{CMGHL24} for our case $c=6$ (Lemma \ref{lem:Luna_slice}) and $|G_X|=|G_F|=2$ (Lemma \ref{lem:stabilizers}) in their notation.
Note that there is no divisorial locus having a strictly bigger stabilizer than $G_X$.

\subsection{Proof of Theorem \ref{mainthm:not_K_equiv}}

We can now prove that these two compactifications are not $K$-equivalent.
\begin{thm}
\label{thm:not_K_equiv}
The compactifications $\M^{\K}$ and $\overline{\B^5/\Gamma}^{\tor}$ are not $K$-equivalent.
\end{thm}
\begin{proof}
It suffices to show that $K_{\M^{\K}}^5\neq K_{\overline{\B^5/\Gamma}^{\tor}}^5$.
By (\ref{nonord_tor_K}) and (\ref{nonord_Kirwan_K}), we need to show that
\[(5\Delta)^5\neq (7T)^5.\]
Now, $T_{\ord,i}^5=6$ by Proposition \ref{prop:normal_bundles}.
Hence, we have $T_{\ord}^5=210$ and
    \[T^5=\frac{210}{8!}=\frac{1}{192}.\]

Here, if $(5\Delta)^5$ and $(7T)^5$ are equal, then the denominator of $\Delta^5$ must be divided by $5$ from the above calculation.
On the other hand, \cite[Proposition 6.10]{CMGHL24} implies 
\[\Delta^5\in\frac{1}{e}\Z,\]
where $e$ is the least common multiple of the orders of $S_x$ for any $x\in\Delta$.
However, the quantity $e$ is not divisible by 5 by Proposition \ref{prop:e_not_divisible_by_5}.
This contradicts to the above.
\end{proof}
\begin{rem}
\label{rem:compdelta}
In principle, there is also a way to compute $\Delta^5$ explicitly. This can be approached as in \cite[Proposition 6.1]{CMGH24}. Using the methods of \cite[Lemma 8.2]{CMGHL24} one can exhibit $\Delta$ 
as a finite quotient of a toric variety. After identifying the normal bundle, one can then compute its top intersection number by toric methods. As we do not need the precise number, we do not pursue 
this lengthy computation.
\end{rem}

\subsection{Relation to the minimal model program}
\label{subsection:mmp}
    We gave a proof of Theorems \ref{mainthm:extendability} and \ref{mainthm:not_K_equiv} by a specific computation in our situation.
    In this subsection, we would like to explain a more systematic approach to relate this to the minimal model program; see \cite{KM98, Fu17} for the basic definitions. This also clarifies the relationship between our work and 
    semi-toric compactifications of arithmetic quotients of type I and IV domains, in the sense of Looijenga; see \cite{Lo86, Lo03a, Lo03b} and \cite{AE23, Od22}.
    We explain the strategy below.

    Let us recall the basic definition of the minimal model program before getting to the discussion.
    We use the notion of \textit{canonical model} in the sense of \cite[Definition 3.50]{KM98}, which is referred to as \textit{log canonical model} in \cite[Definition 4.8.1]{Fu17}.
    In addition, in this paper, \textit{minimal model} is defined as not imposing dlt on \cite[Definition 3.50]{KM98}, which is precisely the \textit{(log) minimal model} in \cite[Definition 4.3.1]{Fu17}.
    We shall first apply the general theory of several compactifications to our case.

On the one hand, by the construction of the Baily-Borel compactifications, the automorphic line bundle  $\L=K_{\overline{\B^5/\Gamma}^{\BB}} + \frac{3}{4}\H$ is ample, thus the pair $(\overline{\B^5/\Gamma}^{\BB}, \frac{3}{4}\H)$ is a canonical model and a good minimal model (of itself).
Here, {\em good} means that $K_{\overline{\B^5/\Gamma}^{\BB}} + \frac{3}{4}\H$ is semi-ample (it is in fact ample). 
The map 
\begin{align}
\label{mor:log_crepant_mumford}
    \pi:\Bigl(\overline{\B^5/\Gamma}^{\tor}, \frac{3}{4}\widetilde{\H}+T\Bigr)\to \Bigl(\overline{\B^5/\Gamma}^{\BB}, \frac{3}{4}\H\Bigr)
\end{align}
is log crepant by \cite[Proposition 3.4]{Mu77}; see also \cite[Theorem 5.3]{HKM24}.
This implies that $(\overline{\B^5/\Gamma}^{\tor}, \frac{3}{4}\widetilde{\H}+T)$ is a minimal model (of itself), and more strongly, a good minimal model with quasi-divisorially log terminal singularities, see \cite[Theorem 3.1 (ii)]{Od22}.
Here we note that the fan, being $1$-dimensional, is automatically regular.
On the other hand, let us consider the Kirwan blow-up.
Here, the problem is whether $(\M^{\K}, \frac{3}{4}\widetilde{\D}+\Delta)$ is a minimal model or not.
For this we need to compute the discrepancy $a(\Delta, \M^{\GIT}, \frac{3}{4}\D)\in\Q$ of $\Delta$ with respect to the Kirwan blow-up $f:(\M^{\K},\frac{3}{4}\widetilde{\D}+\Delta)\to(\M^{\GIT},\frac{3}{4}\D)$:
\begin{align}
    \label{eq:f_discrepancy}
    K_{{\M^K}}=f^*\Bigl(K_{\M^{\GIT}}+\frac{3}{4}\D\Bigr)+a\Bigl(\Delta, \M^{\GIT}, \frac{3}{4}\D\Bigr)\Delta-\frac{3}{4}\widetilde{\D}.
\end{align}
\begin{prop}
\label{prop:mmp}
 The discrepancy $a(\Delta, \M^{\GIT}, \frac{3}{4}\D)$ is $\frac{1}{2}$.
\end{prop}
\begin{proof}
To compute the quantity $a(\Delta, \M^{\GIT}, \frac{3}{4}\D)$ one can use \cite[Remark 6.7]{CMGHL24}, which reduces the problem to the calculation of the discrepancy $a'(\Delta, (\P^8)^{\ss}, \frac{3}{4}\D)$ of $\Delta$ with respect to the map  
    \[f':\Bigl(\widetilde{\P^8}^{\ss},\frac{3}{4}\widetilde{\D}+\Delta\Bigr)\to\Bigl((\P^8)^{\ss},\frac{3}{4}\D\Bigr)\]
    (here we use the notation analogous to \cite[Subsection 6.2]{CMGHL24}).
    Note that since the codimension of the strictly semistable loci $(\P^8)^{\ss}\setminus (\P^8)^{\mathrm{s}}$ and $\M^{\GIT}\setminus (\M^{\GIT})^o$ is larger than 2, we can 
    apply the  \cite[Remark 6.7]{CMGHL24} to our setting.
    In fact, combining this with the computation of (\ref{nonord_Kirwan_K}), we have $a'(\Delta, (\P^8)^{\ss}, \frac{3}{4}\D)=a(\Delta, \M^{\GIT}, \frac{3}{4}\D)$.

We claim that $a'(\Delta, (\P^8)^{\ss}, \frac{3}{4}\D)=\frac{1}{2}$. This follows from the computations in the proof of Theorem \ref{thm:nonord_nontransversal}. 
For this purpose, it suffices to consider an affine locus $\mathcal{P}\defeq(S_0\neq 0)$.
We claim that (\ref{eq:Luna_slice_P}) and the subsequent calculations imply 
    \[f'^*(\D) = \widetilde{\D} + 6\Delta.\]
For this, we must take the involution $\SS_2$, which interchanges the two local analytic branches of the discriminant $V=V_1 \cup V_2$, into account.
Locally analytically, we are in the following situation. We have a commutative diagram       
\begin{equation*}
\xymatrix{
\widetilde X\ar[d]_{\sigma_X}\ar[r]^{\widetilde \pi}& \tilde Y \ar[d]^{\sigma_Y}.  \\
X\ar[r]^{\pi}&Y
}
\end{equation*}
The varieties in this diagram are defined as follows. Here $X$ is the $6$-dimensional Luna slice. 
Further, $\sigma_X: \widetilde X \to X$ is the blow-up of $X$ in the origin.  By $W=W_1 \cup W_2,$ we denote the intersection of the discriminant $V$ with the  Luna slice.
The strict transform
of $W$ will be denoted by $\widetilde W$. The horizontal maps in this diagram are the quotient maps given by the $S_2$-action and $\sigma_Y$ is induced from $\sigma_X$. We denote the image of $W$ in
$Y$ by $Z$ and similarly for $\widetilde W$. Finally, we denote the exceptional divisor in $\widetilde X$ by $E_X$ and its image in $\widetilde Y$ by $E_Y$. We claim that the quotient map
$\widetilde \pi$ is  not ramified along $\widetilde W$ and $E_X$. Indeed, the action of  $\SS_2$ interchanges the two branches of $W$ and hence there is no branching along $\widetilde W$.
Moreover, the local Luna slice calculation also shows that  $\SS_2$ acts non-trivially on the projectivized normal space of the origin
and thus also on the exceptional divisor $E_X$.

   Now, let 
    \[ \sigma_X^*(W)=\widetilde W  + b_X E_X \]
  and
    \[ \sigma_{Y}^*(Z)=\widetilde Z  + b_Y E_Y. \]
We claim that $b_X=b_Y$. Indeed, we have
      \begin{equation*}
      \widetilde \pi^*  \sigma_Y^*(Z)= \widetilde \pi^* (\widetilde Z + b_Y E_Y)= \widetilde W + b_Y E_X
      \end{equation*}
 where we used that $\widetilde \pi^*(E_Y)=E_X$, since $\widetilde \pi$ is not ramified along $E_X$.  Comparing this with 
      \begin{equation*}
      \sigma_X^*  \pi^*(Z)=  \sigma_X^* (W)= \widetilde W + b_X E_X
      \end{equation*}
 gives $b_X=b_Y$. 
 We finally recall that the calculations in the proof of Theorem \ref{thm:nonord_nontransversal} imply that the stabilizer group of a generic point of the intersection of the strict transform of the discriminant and the exceptional divisor is of order $2$ and generated by the involution considered above.   
 
 The claim about $a(\Delta, \M^{\GIT}, \frac{3}{4}\D)$ now follows by combining the above calculation with the equality 
    \[K_{\widetilde{\P^8}^{\ss}} = f'^*(K_{(\P^8)^{\ss}}) + 5\Delta\]
    from (\ref{nonord_Kirwan_K}). A straightforward calculation shows that
    \[K_{\widetilde{\P^8}^{\ss}}  = f'^*\Bigl(K_{(\P^8)^{\ss}} + \frac{3}{4}\D\Bigr) + \frac{1}{2}\Delta - \frac{3}{4}\widetilde{\D}\]
    and hence $a(\Delta, \M^{\GIT}, \frac{3}{4}\D)=\frac{1}{2}$.
\end{proof}

    One can extend the notion of $K$-equivalence to pairs.
    Let $(X, \Delta_X)$ and $(Y,\Delta_Y)$ be pairs of    projective normal $\Q$-Gorenstein varieties $X$ and $Y$ and $\Q$-divisors $\Delta_X\in\Pic(X)$ and $\Delta_Y\in\Pic(Y)$ with a birational morphism $g:X \dashrightarrow Y$ and $g_*\Delta_X = \Delta_Y$. 
    We call these pairs \textit{$K$-equivalent as pairs} if 
  there is a common resolution of singularities $Z$ dominating $X$ and $Y$ birationally such that $f_X:Z\to X$ and $f_Y:Z\to Y$ satisfy $f_X^*(K_X+\Delta_X) \sim_{\Q} f_Y^*(K_Y+\Delta_Y)$.
    The above calculation implies the following proposition, which can be seen as a variation of Theorem \ref{mainthm:not_K_equiv}.
    \begin{prop}\label{prop:notsemitoric}
 $(\M^{\K}, \frac{3}{4}\widetilde{\D}+\Delta)$ and  $(\overline{\B^5/\Gamma}^{\tor}, \frac{3}{4}\widetilde{\H}+T)$ are not $K$-equivalent as pairs.
    \end{prop}
\begin{proof}
We give a proof by contradiction. 
    We first observe that formula (\ref{eq:f_discrepancy}), combined with Proposition \ref{prop:mmp}, shows that
    \begin{equation}
     \label{equ:comparepullbackK1}
    K_{\M^{\K}} + \frac{3}{4}\widetilde{\D} + \Delta = f^*(K_{\M^{\GIT}} + \frac{3}{4}\D) + \frac{3}{2}\Delta. 
   \end{equation}
   Now assume that  $(\M^{\K}, \frac{3}{4}\widetilde{\D}+\Delta)$ and  $(\overline{\B^5/\Gamma}^{\tor}, \frac{3}{4}\widetilde{\H}+T)$ are $K$-equivalent as pairs.
   Then there exists a common resolution $Z$ with birational morphisms $f_1:Z \to \M^{\K}$ and $f_2:Z \to \overline{\B^5/\Gamma}^{\tor}$ such that
   \begin{equation}
     \label{equ:comparepullbackK2}
   f_1^*(K_{\M^{\K}} + \frac{3}{4}\widetilde{\D} + \Delta) = f_2^*(K_{\overline{\B^5/\Gamma}^{\tor}} + \frac{3}{4}\widetilde{\H} + T).
   \end{equation}
   Altogether, we have a commutative diagram 
   \begin{equation*}
\xymatrix{
&Z\ar[ld]_{f_1} \ar[rd]^{f_2}&\\
\M^{\K}\ar[d]^f \ar@{-->}[rr]^g&&\overline{\B^5/\Gamma}^{\tor} \ar[d]^{\pi}\\
\M^{\GIT} \ar[rr]^{\phi}& & \overline{\B^5/\Gamma}^{\BB}.
}
\end{equation*}
Using this, we obtain
   \begin{equation}
   \label{equ:comparepullbackK3}
   f_1^*(f^*(K_{\M^{\GIT}} + \frac{3}{4}\D))=f_2^*(\pi^*((\phi^{-1})^*(K_{\M^{\GIT}}+ \frac{3}{4}\D)))=f_2^*(K_{\overline{\B^5/\Gamma}^{\tor}} + \frac{3}{4}\widetilde{\H} + T).
   \end{equation}
   Combining formulae (\ref{equ:comparepullbackK1}), (\ref{equ:comparepullbackK2}) and  (\ref{equ:comparepullbackK3}) implies that $f_1^*(\Delta)=0$ in $\Pic(Z)\otimes\Q$.
   However, since $f_1$ is proper birational and $\M^{\K}$ is normal,  the Stein factorization implies that$f_{1,*}\OO_Z = \OO_{\M^{\K}}$.
   Now, let $L$ be a line bundle on $\M^{\K}$.
   We assume that $f_1^*L = \OO_Z$.
   Then, by the projection formula, we have $f_{1,*}f_1^*L = L \otimes f_{1,*}\OO_Z$.
   Our assumption and the above observation imply that $L\cong \OO_{\M^{\K}}$, which results in showing that the pullback $f_1^*:\Pic(\M^{\K}) \to \Pic (Z)$ is injective.

\end{proof}
\begin{rem}
  For the ordered case, Gallardo-Kerr-Schaffler \cite[Theorem 2.1]{GKS21} showed that $\M^{\K}_{\ord}\cong\overline{\B^5/\Gamma_{\ord}}^{\tor}$ through the natural morphism lifting the Deligne-Mostow isomorphism.
  Here we remark that our computation determining the discrepancies can be used to give a different proof of this statement. The argument goes as follows.
First, we prove that the pair $(\M_{\ord}^{\K},\frac{1}{2}\widetilde{\D_{\ord}} + \Delta_{\ord})$ 
is a log minimal model of itself with log terminal singularities, as is $(\overline{\B^5/\Gamma_{\ord}}^{\tor},\frac{1}{2}\widetilde{\H_{\ord}} + T_{\ord})$, which can be shown using the same 
arguments as in \cite[Proposition 5.1, Theorem 5.3]{HKM24}.
  The discrepancy of the birational morphism  
    \[\varphi_1:\Bigl(\M_{\ord}^{\K}, \frac{1}{2}\widetilde{\D_{\ord}}+\Delta_{\ord}\Bigr)\to\Bigl(\M^{\GIT}_{\ord}, \frac{1}{2}\D_{\ord}\Bigr)\]
    is $-1$ by (\ref{eq:pull_back}) and (\ref{eq:can_bindle_M^K_{ord}}).
    More precisely, these formulae show that
    \[K_{\M^{\K}_{\ord}}=\varphi_1^*(K_{\M^{\GIT}_{\ord}}+\frac{1}{2}\D_{\ord})-\Delta_{\ord}-\frac{1}{2}\widetilde{\D_{\ord}},\]
    which implies that 
    \[\phi_{\ord}\circ\varphi_1:\left(\M_{\ord}^{\K},\frac{1}{2}\D_{\ord} + \Delta_{\ord}\right) \to \left(\overline{\B^5/\Gamma_{\ord}}^{\BB}, \frac{1}{2}\H_{\ord}\right)\] is log crepant. 
   This is the same situation as the blow-up
\[\pi_{\ord}:\left(\overline{\B^5/\Gamma_{\ord}}^{\tor},\frac{1}{2}\widetilde{\H_{\ord}} + T_{\ord}\right) \to \left(\overline{\B^5/\Gamma_{\ord}}^{\BB}, \frac{1}{2}\H_{\ord}\right);\]
see also (\ref{mor:log_crepant_mumford}).
Hence, a similar proof to \cite[Proposition 5.1 (3)]{HKM24} implies that the pair $(\M_{\ord}^{\K},\frac{1}{2}\D_{\ord} + \Delta_{\ord})$ has log terminal singularities.
To prove that it is a log minimal model of itself, we check that it satisfies \cite[Definition 4.3.1]{Fu17}.
\cite[Definition 4.3.1 (i)-(iii)]{Fu17} is trivially satisfied as $(X,\Delta) = (X', \Delta') = (\M_{\ord}^{\K}, \frac{1}{2}\widetilde{\D_{\ord}}+\Delta_{\ord})$ in the notation there.
Since \[6\L_{\ord} = K_{\overline{\B^5/\Gamma_{\ord}}^{\BB}} + \frac{1}{2}\H_{\ord} \]
is ample and $\phi_{\ord}\circ\varphi_1$ is log crepant, the $\Q$-line bundle
\[K_{\M_{\ord}^{\K}} + \frac{1}{2}\D_{\ord} + \Delta_{\ord}\]
is nef, which shows that the pair $(\M_{\ord}^{\K}, \frac{1}{2}\widetilde{\D_{\ord}}+\Delta_{\ord})$ satisfies \cite[Definition 4.3.1 (iv)]{Fu17}.
\cite[Definition 4.3.1 (v)]{Fu17} follows from the above computation, showing $\phi_{\ord}\circ\varphi_1$ is log crepant; see also \cite[Theorem 5.3]{HKM24}.
Summarizing the above, the pair $(\M_{\ord}^{\K}, \frac{1}{2}\widetilde{\D_{\ord}}+\Delta_{\ord})$ is a log minimal model with log terminal singularities.
   We can now finish the claim concerning the isomorphism $\M^{\K}_{\ord}\cong\overline{\B^5/\Gamma_{\ord}}^{\tor}$.
   According to \cite[Theorem 3.1]{Od22}, it remains to prove that $\D_{\mathrm{ord}}\cup \Delta_{\mathrm{ord}}$ is a normal crossing divisor.
 This follows from a straightforward explicit Luna slice computation; see also Remark \ref{rem:orderediso}. We omit the details.
\end{rem}

Finally, we shall refer to a characterization as a semi-toric compactification.
\begin{prop}
\begin{enumerate}
    \item The pair $(\M^{\K}, \frac{3}{4}\widetilde{\D}+\Delta)$ is not log canonical.
    \item $\M^{\K}$ is not a semi-toric compactification.
    \end{enumerate}
\end{prop}\begin{proof}
    Odaka gave a characterization of semi-toric compactifications in terms of singularities of pairs \cite[Theorem 3.1 (iii)]{Od22}.
    From this we can deduce that (1) implies (2). Hence it suffices to prove the first of the two statements.   
    This can be shown in the same way as \cite[Cororally 5.10]{HKM24}. We omit the details, but a sketch of the proof is as follows.
    Essentially, it follows from the fact that the log canonical centre of the Baily-Borel compactification is the Baily-Borel (unique in this case) cusp \cite[Lemma 2.9 (1)]{MO23}.
    This implies that an exceptional divisor of a resolution of singularities of $\M^{\K}$ is mapped to the unique Baily-Borel cusp $\xi$ in $\overline{\B^5/\Gamma_{\ord}}^{\BB}$.
    Combined with the inequality in \cite[Lemma 2.27]{KM98}, we can deduce that the discrepancy around the exceptional divisor is strictly less than $-1$.
    This concludes the proof.
\end{proof}

\begin{rem}\label{rem:newproof13}
\begin{enumerate}
    \item 
We note that this gives another proof of one direction of Theorem \ref{mainthm:extendability}. Indeed, this can be deduced from \cite[Theorem 7.18]{AE23} or \cite[Subsection 3F]{AEC24}, since semi-toric
compactifications are characterized by the property that they lie between $\overline{\B^5/\Gamma}^{\tor}$ and $\overline{\B^5/\Gamma}^{\BB}$. 
Note that in this case, one does not need to know the 
equality of the second Betti numbers of $\M^{\K}$ and $\overline{\B^5/\Gamma}^{\tor}$ which we used in our proof.
\item The phenomenon that the pair, consisting of the Kirwan blow-up with the exceptional divisor and the discriminant divisor with standard coefficients 
has worse than log canonical singularities, so that, in particular, it is not a semi-toric compactification of ball a quotient, can be observed in other situations as well.
One example is the case of 12 points, which corresponds to the Eisenstein ancestral Deligne-Mostow variety. This was studied in \cite[Corollary 5.10]{HKM24}.
\end{enumerate}
\end{rem}

\section{Cohomology}\label{sec:Bettinumbers}
In this section, we compute the cohomology of the varieties appearing in this paper.

\subsection{The cohomology of $\M_{\ord}^{\K}$,  $\overline{\B^5/\Gamma_{\ord}}^{\BB}$, $\overline{\B^5/\Gamma_{\ord}}^{\tor}$ and $\overline{\B^5/\Gamma}^{\BB}$}
We first collect the results due to Kirwan-Lee-Weintraub \cite{KLW87} and Kirwan \cite{Ki89}
who determined the Betti numbers of $\M_{\ord}^{\K}$ and $\overline{\B^5/\Gamma_{\ord}}^{\BB}$,  
and $\M^{\GIT}\cong \overline{\B^5/\Gamma}^{\BB}$ respectively. We summarize this in

\begin{thm}[{\cite[Table III, Theorem 8.6]{KLW87}}, {\cite[Table, p.40]{Ki89}}]
\label{thm:coh_previous_work}
All the odd degree cohomology of $\M_{\ord}^{\K}$,  $\overline{\B^5/\Gamma_{\ord}}^{\BB}$ and $\overline{\B^5/\Gamma}^{\BB}$ vanishes.
In even degrees, the Betti numbers are as follows:

\begin{align*}
\renewcommand*{\arraystretch}{1.2}
\begin{array}{l|cccccc}
\hskip2cm j&0&2&4&6&8&10\\\hline
\dim H^j(\M_{\ord}^{\K})&1&43&99&99&43&1\\
\dim IH^j(\overline{\B^5/\Gamma_{\ord}}^{\BB})&1&8&29&29&8&1\\
\dim IH^j(\M^{\GIT})&1&1&2&2&1&1\\
\dim IH^j(\overline{\B^5/\Gamma}^{\BB})&1&1&2&2&1&1\\
\end{array}
\end{align*}
\end{thm}

By an application of an easy version of the decomposition theorem, we can also compute the cohomology of $\overline{\B^5/\Gamma_{\ord}}^{\tor}$ (without using that this space is isomorphic 
to $\M_{\ord}^{\K}$). 

\begin{thm}
\label{thm:coh_ordered_tor}
All the odd degree cohomology of $\overline{\B^5/\Gamma_{\ord}}^{\tor}$ vanishes.
In even degrees, the Betti numbers are as follows:

\begin{align*}
\renewcommand*{\arraystretch}{1.2}
\begin{array}{l|cccccc}
\hskip2cm j&0&2&4&6&8&10\\\hline
\dim H^j(\overline{\B^5/\Gamma_{\ord}}^{\tor})&1&43&99&99&43&1\\
\end{array}
\end{align*}
\end{thm}
\begin{proof}
We use the form of the decomposition theorem as given in \cite[Lemma 9.1]{GH17}. Here we have $35$ cusps and the toroidal boundary at each cusp is isomorphic to $\P^2 \times \P^2$. The even
Betti numbers of this space are given by $(1,2,3,2,1)$ and the result then follows from the Betti numbers of $\overline{\B^5/\Gamma_{\ord}}^{\BB}$ together with the fact that there are 35 cusps.   
\end{proof}

\subsection{The cohomology of $\M^{\K}$}
Now, we compute the cohomology of $\M^{\K}$.
This will be done using the Kirwan method \cite{Ki84, Ki85, Ki89}, studying the cohomology of the Kirwan blow-ups.
We mainly follow \cite[Chapter 3, 4]{CMGHL23}, in particular, the case of cubic threefolds with precisely $2A_5$-singularities.
Let us consider $X=\P^8$, acted on by $G=\SL_2(\C)$ with the usual linearization and let $Z_R^{\ss}$ be the fixed locus of the action of $R$ on $X^{\ss}$, which is the semi-stable locus.
Here we recall that $R$ is the stabilizer of the strictly semi-stable point $c_{4,4}$ as introduced in Lemma \ref{lem:stabilizers}.
We denote by $\widetilde{X}^{\ss}\defeq Bl_{G\cdot Z_R^{\ss}}(X)$ the blow-up whose centre
is the unique polystable orbit $G\cdot Z_R^{\ss}$. 
From \cite[Section 3 Eq. 3.2]{Ki89} or \cite[Subsection 4.12, (4.22)]{CMGHL23}, 
the Poincar\'e series of $\widetilde{X}^{\ss}$ is given by 
\[P_t^G(\widetilde{X}^{\ss})=P_t^G(X^{\ss})+A_R(t),\]
where $A_R(t)$ is a correction term consisting  of a ``main term" and an  ``extra term" with respect to the unique stabilizer $R$; see \cite[Section 4.1.2]{CMGHL23} for precise definitions.

This method reduces the computation of $H^k(\M^{\K})$ to the estimation of  
\begin{enumerate}
    \item the semi-stable locus (Subsection \ref{subsection:semi-stable}) ,
    \item the main correction term (Subsection \ref{subsection:main_correction}) and  
    \item the extra correction term (Subsection \ref{subsection:extra_correction}).
\end{enumerate}

\subsubsection{$\mathrm{\mathbf{Equivariant \ cohomology\ of\ the\ semi\mathchar`-stable\ locus}}$}
\label{subsection:semi-stable}

Here we proceed according to \cite[Chapter 3]{CMGHL23}.
We can compute the cohomology of the semi-stable locus by using the stratification introduced by Kirwan.
We omit details, but will still need to introduce some notation in order to describe the outline. 
Let $\{S_{\beta}\}_{\beta\in\mathcal{B}}$ be the stratification defined in \cite[Theorem 4.16]{Ki84} and $d(\beta)$ be the codimension of $S_{\beta}$ in $X^{\ss}$.
Here, the index set $\mathcal{B}$ consists of the point which is closest to the origin of the convex hull spanned by some weights in the closure of a positive Weyl chamber in the Lie algebra of a maximal torus in $\SO(2)$; see \cite[Chapter 3]{CMGHL23} or \cite[Definition 3.13]{Ki84} for details.

\begin{prop}
\label{prop:semi-stable_locus}
\[P_t^G(X^{\ss})\equiv 1+t^2+2t^4 \bmod t^6.\]
\end{prop}
\begin{proof}
We shall prove $2d(\beta)\geq 6$ for any $0\neq\beta\in\mathcal{B}$.
This implies
\begin{align*}
    P_t^G(X^{\ss})&\equiv P_t(X)P_t(B\SL_2(\C))\bmod t^6 \\
        &\equiv (1-t^2)^{-1}(1-t^4)^{-1}\bmod t^6\\
    &\equiv 1+t^2+2t^4\bmod t^6.
\end{align*}
Here we denote by $BG$ the classifying space for any topological space $G$; see \cite[Appendix A]{CMGHL23}.
In the same way, as in the proof of \cite[Proposition 3.5]{CMGHL23} we obtain 
\[d(\beta)\geq7-r(\beta),\]
where $r(\beta)$ is the number of weights $\a$ satisfying $\beta\cdot\a\geq ||\beta||^2$.
Now, we have
\[\mathcal{B}=\{(1,-1), (2,-2), (3,-3), (4,-4)\}.\]
For each $(a,-a)\in\mathcal{B}$, it easily follows 
\[r(\beta)=5-a,\]
and this implies $d(\beta)\geq 3$.

\end{proof}

\subsubsection{$\mathrm{\mathbf{The\ main\ correction\ term}}$}
\label{subsection:main_correction}

The following is based on \cite[Chapter 4]{CMGHL23}.

\begin{prop}
\label{prop:main_correction}
The main correction term in $A_R(t)$ is given by 

 \[ (1-t^4)^{-1}(t^2+t^4)\equiv t^2+t^4 \bmod t^6.  \]

\end{prop}
\begin{proof}
In the same way as in \cite[Proposition B.1 (4)]{CMGHL23}, the normalizer of $R$ is computed to be
\[N\defeq N(R)\cong \mathbb{T}\rtimes \Z/2\Z.\]
Hence, it follows that
\begin{align*}
    H_N^{\bullet}(Z_R^{\ss})&=(H_{\mathbb{T}}^{\bullet}(Z_R^{\ss}))^{\Z/2\Z}\\
    &=(H^{\bullet}(BR)\otimes H^{\bullet}_{\mathbb{T}/R}(Z_R^{\ss}))^{\Z/2\Z}\\
    &=(H^{\bullet}(BR)\otimes H^{\bullet}(*))^{\Z/2\Z}\\
        &=\Q[c^4]
\end{align*}
where $*$ denotes a set of 1 point and the degree of $c$ is 1.
The last equation follows from the discussion in the proof of \cite[Proposition 4.4]{CMGHL23}.
Hence, 
\[P_t^N(Z_R^{\ss})=(1-t^4)^{-1}.\]
Combining this with \cite[(4.24)]{CMGHL23} completes the proof.
\end{proof}

\subsubsection{$\mathrm{\mathbf{The\ extra\ correction\ term}}$}
\label{subsection:extra_correction}

Let $\mathcal{N}$ be the normal bundle to the orbit $G\cdot Z_R^{\ss}$.
Then, for a generic point $x\in Z_R^{\ss}$, we have a representation $\rho$ of $R$ on $\mathcal{N}_x$.
Let $\mathcal{B}(\rho)$ be the set consisting of the closest point to 0 of the convex hull of a nonempty set of weights of the representation $\rho$.
For $\beta'\in\mathcal{B}(\rho)$, let $n(\beta')$ be the number of weights less than $\beta'$.

\begin{prop}
\label{prop:extra_correction}
The extra correction term vanishes modulo $t^6$, i.e.,  does not contribute to $A_R(t)$.
\end{prop}
\begin{proof}
In our case we have $Z_R^{\ss}=\{c_{4,4}\}$.
Thus, to describe $\mathcal{N}_x$, we have to compute  \[\Bigl(T_{c_{4,4}}(\SL_2(\C)\cdot\{c_{4,4}\})\Bigr)^{\perp}.\]
This was calculated in Lemma \ref{lem:Luna_slice}.
Moreover, $\diag(\lambda, \lambda^{-1})$ acts on $T_{c_{4,4}}\C^9\cong\C^9$ by the weights 
\[0, \pm 2, \pm 4, \pm 6, \pm 8.\]
It follows that $T_{c_{4,4}}(\SL_2(\C)\cdot\{c_{4,4}\})$ is generated by the weights $\{0, \pm 2\}$, and hence we obtain
\[\mathcal{B}(\rho)=\{\pm 4, \pm 6, \pm 8\}.\]
This shows that 
\begin{align*}
    d(|\beta'|)&=n(|\beta'|)\\
    &=1+\frac{|\beta'|}{2}\\
    &\geq 3
\end{align*}
for $\beta'\in\mathcal{B}(\rho)$.
This in turn implies that 
\[\mathrm{``extra\ correction\ term"} \equiv 0 \bmod t^6\]
by \cite[(4.25)]{CMGHL23}.
\end{proof}

\subsubsection{$\mathrm{\mathbf{Computation\ of\ the\ cohomology\ of\  \M^{\K}}}$}
From Propositions \ref{prop:semi-stable_locus}, \ref{prop:main_correction} and \ref{prop:extra_correction}, it follows that
\begin{align*}
    P_t(\M^{\K})&= P_t^G(\widetilde{X}^{\ss})\\
    &\equiv (1+t^2+ 2t^4)+(t^2+t^4) \bmod t^6\\
    &\equiv 1+2t^2+3t^4 \bmod t^6.
\end{align*}
Therefore, we obtain the following:

\begin{thm}
\label{thm:coh_M^K}
All the odd degree cohomology of $\M^{\K}$ vanishes.
In even degrees, its Betti numbers are given as follows:
\begin{align*}
\renewcommand*{\arraystretch}{1.2}
\begin{array}{l|cccccc}
\hskip2cm j&0&2&4&6&8&10\\\hline
\dim H^j(\M^{\K})&1&2&3&3&2&1\\
\end{array}
\end{align*}
\end{thm}

\subsection{The cohomology of $\overline{\B^5/\Gamma}^{\tor}$}
Now, we compute the cohomology of the toroidal compactification of the 5-dimensional ball quotient.
Our main tool is the decomposition theorem in the easy form stated in theorem \cite[Lemma 9.1]{GH17}, see also \cite[chapter 6]{CMGHL23}.
This allows us to combine the cohomology of $\overline{\B^5/\Gamma}^{\BB}$ and the toroidal boundary.
To do this, we first study the cohomology of the toroidal boundary.

\begin{prop}
\label{prop:coh_boundary}
All the odd degree cohomology of the boundary $T$ vanishes.
In even degrees, its Betti numbers are given as follows:
\begin{align*}
\renewcommand*{\arraystretch}{1.2}
\begin{array}{l|ccccc}
\hskip2cm j&0&2&4&6&8\\\hline
\dim H^j(T)&1&1&2&1&1\\
\end{array}
\end{align*}
\end{prop}

\begin{proof}
This amounts to the computation of the invariant cohomology of the action of the stabilizer of a toroidal boundary component as in the proof of \cite[Proposition 7.13]{CMGHL23}.
More precisely, we have to determine the cohomology ring 
\[
 H^{\bullet}(\P^2 \times \P^2)^{(\SS_4\times\SS_4)\rtimes\SS_2} = H^{\bullet}((\P^2/\SS_4)^2, \Q)^{\SS_2} = H^{\bullet}((\P(1,2,3)^2, \Q)^{\SS_2} .
\]
Since $H^{\bullet}(\P^2/\SS_4) =   H^{\bullet}((\P(1,2,3))  \cong \Q[x]/(x^3)$, this is equivalent to compute the $\SS_2$-invariant parts of the tensor product $\Q[x]/(x^3)\otimes\Q[y]/(y^3)$.
Hence the invariant cohomology is given by 
\[P_t(T)=1+t^2+2t^4+t^6+t^8.\]
\end{proof}

We can now summarize the above computations in the 

\begin{thm}
\label{thm:coh_tor}
All the odd degree cohomology of $\overline{\B^5/\Gamma}^{\tor}$ vanishes.
In even degrees, the Betti numbers are given by the following table:
\begin{align*}
\renewcommand*{\arraystretch}{1.2}
\begin{array}{l|cccccc}
\hskip2cm j&0&2&4&6&8&10\\\hline
\dim H^j(\overline{\B^5/\Gamma}^{\tor})&1&2&3&3&2&1\\
\end{array}
\end{align*}
In particular, all the Betti numbers of $\M^{\K}$ and $\overline{\B^5/\Gamma}^{\tor}$ are the  same.
\end{thm}
\begin{proof}
This follows now from an application of the decomposition theorem as stated in \cite[Lemma 9.1]{GH17}, applied to the last line in Theorem \ref{thm:coh_previous_work} and Proposition \ref{prop:coh_boundary}.  
\end{proof}

\section{Other cases of the Deligne-Mostow list}\label{sec:othercases}
Here we very briefly discuss some further cases of the Deligne-Mostow list where a similar analysis can be made.
More concretely, we consider $N$ points on $\P^1$ for $5\leq N \leq 12$ with symmetric weights; see \cite{DM86} or \cite[Appendix]{Th98}.
Note that the notions of stable and semi-stable coincide for odd $N$.
Remarkably, the beahviour which was observed for the moduli spaces of cubic surfaces and 8 points on $\P^1$, can also be found in other cases, thus pointing towards a 
much more general phenomenon. 

\subsection{5 points}
The moduli space of 5 points on $\P^1$ is associated with K3 surfaces with an automorphism of order 5 \cite{Ko07b}.
In this case, the Deligne-Mostow isomorphism gives 
\[\M_{\ord}^{\GIT}\cong\overline{\B^2/\Gamma_{\ord}}^{\BB}\]
for the discriminant kernel group $\Gamma_{\ord}$ \cite[Subsection 6.3, (6.5)]{Ko07b}.
Here, the weight in the sense of Deligne-Mostow is
\[\Bigl(\frac{2}{5},\frac{2}{5},\frac{2}{5},\frac{2}{5},\frac{2}{5}\Bigr).\]
This is the quintic del Pezzo surface \cite[Proposition 6.2 (2)]{Ko02}.
Now, $\B^2/\Gamma'$ is compact (\cite[Subsection 6.5]{Ko07b} or \cite[Appendix]{Th98}).
Hence, we have 
\[\M^{\K}=\M^{\GIT}\cong\overline{\B^2/\Gamma}^{\BB}=\overline{\B^2/\Gamma}^{\tor}\]
for the full modular unitary group $\Gamma$.

\subsection{7, 9, 10 or 11 points}
The moduli space of 7 points on $\P^1$ was studied in \cite{DvGK05}.
In this paper, we apply the theory of the moduli spaces of stable curves to analyse the geometry of our ball quotients.
In order to apply the work by Hassett, Kiem-Moon and others, the weights appearing in the Deligne-Mostow theory, that is the linearization of a line bundle, must be linearised as $\OO(1,\cdots,1)$; see \cite[Section 1]{KM11}.
Thus, in particular, the case of 7, 9, 10, and 11 points are out of scope in this paper.

\subsection{6 points and 12 points}

These are Eisenstein cases, which will be treated in upcoming work. 

\subsubsection{\rm{\textbf{6\ points}}}
The moduli space of 6 points on $\P^1$ is closely related to the theory of the Igusa quartic and the Segre cubic \cite{Ko13, Ko16, Ma01}.
It is known that the Segre cubic is realised as the Baily-Borel compactification of a 3-dimensional ball quotient. We recall the setting of \cite{Ko13}.
Let $\Lambda\defeq \Z[\omega]^{\oplus 4}$ be the Hermitian lattice over $\Z[\omega]$ of signature $(1,3)$ equipped with the Hermitian matrix $\diag(1,-1,-1,-1)$, where $\omega$ is a primitive third root of unity.
Let $\Gamma\defeq\U(\Lambda)(\Z)$ and 
\[\Gamma_{\ord}\defeq\{g\in\Gamma\mid g\vert_{\Lambda/\sqrt{-3}\Lambda}=\id\}.\]
The ball quotient $\overline{\B^3/\Gamma}^{\BB}$ (resp. $\overline{\B^3/\Gamma_{\ord}}^{\BB}$) is isomorphic to the moduli space of unordered (resp. ordered) 6 points on $\P^1$.
Here, $\B^3$ is the 3-dimensional complex ball.
The approach developed in the current paper can be fully carried over to this case. In particular, the analogues of Theorems \ref{mainthm:extendability} and    \ref{mainthm:not_K_equiv}
hold unchanged.

\subsubsection{\rm{\textbf{12 points}}}
The moduli space of unordered 12 points on $\P^1$ is known to be the moduli space of (non-hyperelliptic) curves of genus 4 \cite{Ko02}.
In particular, this moduli space is the $9$-dimensional ball quotient taken by the full unitary group for the Hermitian lattice with underlying integral lattice $U^{\oplus 2} \oplus E_8(-1)^{\oplus 2}$.
There is, however, an important difference here to the cases discussed previously: the arithmetic subgroup defining the moduli space of ordered 12 points on $\P^1$ is not known; see \cite[Remark 2.2]{HKM24}, although it is expected to be the discriminant kernel as in the case of 6 or 8 points.

In this case, there is the blow-up sequence 
\[\overline{\M}_{0,12}\to\overline{\M}_{0,12(\frac{1}{4}+\epsilon)}\to\overline{\M}_{0,12(\frac{1}{5}+\epsilon)}\to\overline{\M}_{0,12(\frac{1}{6}+\epsilon)}\cong\M_{\ord}^{\K}\stackrel{\varphi_1}{\to}\M_{\ord}^{\GIT}.\]
We have since analysed the case of 12 points in some detail, see \cite{HKM24}, showing  
that an analogue of Theorem \ref{mainthm:extendability} also holds in the 12 points case.
This further confirms preliminary computations due to Casalaina-Martin (private communication).


\begin{thebibliography}{99}

\bibitem[AE23]{AE23}
V. Alexeev, P. Engel,
\textit{Compact moduli of K3 surfaces}, 
Ann. Math. (2) 198, No. 2, 727-789 (2023).


\bibitem[AEC24]{AEC24}
V. Alexeev, P. Engel, C. Han, \textit{Compact moduli of K3 surfaces with a nonsymplectic automorphism}, Trans. Am. Math. Soc., Ser. B 11, 144-163 (2024).



\bibitem[AMRT10]{AMRT10}
A. Ash,  D. Mumford, M. Rapoport, Y.-S. Tai, 
\textit{Smooth compactifications of locally symmetric varieties}, 
Second edition. With the collaboration of Peter Scholze. Cambridge Mathematical Library. Cambridge University Press, Cambridge, 2010.

\bibitem[BB66]{BB66}
W. Baily, A. Borel,
\textit{Compactification of arithmetic quotients of bounded symmetric domains},
Ann. of Math. (2) 84 (1966), 442–528.

\bibitem[Be12]{Be12}
N. Behrens,
\textit{Singularities of ball quotients},
Geom. Dedicata 159, 389-407 (2012).

\bibitem[Bo72]{Bo72}
A. Borel,
\textit{Some metric properties of arithmetic quotients of symmetric spaces and an extension theorem},
 J. Differential Geom. 6 (1972), 543--560.
 
 \bibitem[CMGH24]{CMGH24}
S. Casalaina-Martin, S. Grushevsky, K. Hulek,
\textit{The birational geometry of moduli of cubic surfaces and cubic surfaces with a line},
To appear: Moduli (2024).

\bibitem[CMGHL23]{CMGHL23}
S. Casalaina-Martin, S. Grushevsky, K. Hulek, R. Laza,
\textit{Cohomology of the moduli space of cubic threefolds and its smooth models},
Memoirs of the American Mathematical Society 1395 (2023).

\bibitem[CMGHL24]{CMGHL24}
S. Casalaina-Martin, S. Grushevsky, K. Hulek, R. Laza,
\textit{Non-isomorphic smooth compactifications of the moduli space of cubic surfaces}, 
Nagoya Math. J. 254, 315-365 (2024).

\bibitem[DM86]{DM86}
P. Deligne, G. D. Mostow,
\textit{Monodromy of hypergeometric functions and nonlattice integral monodromy},
Inst. Hautes Études Sci. Publ. Math. 63 (1986), 5-89.

\bibitem[DvGK05]{DvGK05}
I. Dolgachev, B. van Geemen, S. Kond\={o},
\textit{A complex ball uniformization of the moduli space of cubic surfaces via periods of K3 surfaces},
J. Reine Angew. Math. 588 (2005), 99-148.


\bibitem[Do88]{Do88}
I. Dolgachev, D. Ortland,
\textit{Point sets in projective spaces and theta functions},
Astérisque No. 165 (1988), 210 pp. (1989).

\bibitem[Do08]{Do08}
I. Dolgachev,
\textit{Reflection groups in algebraic geometry},
Bull. Amer. Math. Soc. (N.S.) 45 (2008), no. 1, 1–60.

\bibitem[FM11]{FM11}
E. Freitag, R.-S. Manni,
\textit{The modular variety of hyperelliptic curves of genus three},
Trans. Amer. Math. Soc. 363 (2011), no. 1, 281–312.

\bibitem[Fu17]{Fu17}
O. Fujino,
\textit{Foundations of the minimal model program},
MSJ Memoirs, 35. Mathematical Society of Japan, Tokyo, 2017.


\bibitem[FM94]{FM94}
W. Fulton, R. MacPherson,
\textit{A compactification of configuration spaces},
Ann. of Math. (2) 139 (1994), no. 1, 183-225.


\bibitem[GKS21]{GKS21}
P. Gallardo, M. Kerr, L. Schaffler,
\textit{Geometric interpretation of toroidal compactifications of moduli of points in the line and cubic surfaces},
Adv. Math. 381 (2021), Paper No. 107632.

\bibitem[GH17]{GH17}
S. Grushevsky, K. Hulek,
\textit{The intersection cohomology of the Satake compactification of ${\A}_g$ for $g \leq 4$},
Math. Ann. 369 (2017), 1353--1381.

\bibitem[Ha03]{Ha03}
B. Hassett,
\textit{Moduli spaces of weighted pointed stable curves},
Adv. Math. 173 (2003), no. 2, 316-352.

\bibitem[He15]{He15}
G. Heckman,
\textit{The Allcock ball quotient},
Pure Appl. Math. Q. 11 (2015), no. 4, 655-681.

\bibitem[HKM24]{HKM24}
K. Hulek, S. Kond\=o, Y. Maeda
\textit{Compactifications of the ancestral Eisenstein Deligne-Mostow variety},
arxiv:2403.18345.

\bibitem[Ki84]{Ki84}
F. C. Kirwan,
\textit{Cohomology of quotients in symplectic and algebraic geometry},
Mathematical Notes, vol. 31, Princeton University Press, Princeton, NJ, 1984.

\bibitem[Ki85]{Ki85}
F. C. Kirwan,
\textit{Partial desingularisations of quotients of nonsingular varieties and their
Betti numbers},
Ann. of Math. (2) 122 (1985), no. 1, 41–85.

\bibitem[Ki89]{Ki89}
F. C. Kirwan,
\textit{Moduli spaces of degree d hypersurfaces in $\P^n$},
Duke Math. J. 58 (1989),
no. 1, 39–78.


\bibitem[KLW87]{KLW87}
F. C. Kirwan, R. Lee, S. H. Weintraub,
\textit{Quotients of the complex ball by discrete groups},
Pacific J. Math. 130 (1987), no. 1, 115-141.

\bibitem[KM11]{KM11}
Y.-H. Kiem, H-.B. Moon,
\textit{Moduli spaces of weighted pointed stable rational curves via GIT},
Osaka J. Math. 48 (2011), no. 4, 1115–1140.

\bibitem[KM98]{KM98}
M. Kollár, S. Mori,
\textit{Birational geometry of algebraic varieties
With the collaboration of C. H. Clemens and A. Corti},
 Cambridge Tracts in Mathematics, 134. Cambridge University Press, Cambridge, 1998.


\bibitem[K\={o}02]{Ko02}
S. Kond\={o},
\textit{The moduli space of curves of genus 4 and Deligne-Mostow's complex reflection groups. Algebraic geometry 2000},
Azumino (Hotaka), 383–400,
Adv. Stud. Pure Math., 36, Math. Soc. Japan, Tokyo, 2002.

\bibitem[Ko07a]{Ko07a}
S. Kond\={o},
\textit{The moduli space of 8 points of $\mathbb{P}^1$ and automorphic forms},
Algebraic geometry, 89–106,
Contemp. Math., 422, Amer. Math. Soc., Providence, RI, 2007.

\bibitem[Ko07b]{Ko07b}
S. Kond\={o},
\textit{The moduli space of 5 points on $\P^1$ and K3 surfaces},
Arithmetic and geometry around hypergeometric functions, 189–206,
Progr. Math., 260, Birkhäuser, Basel, 2007.

\bibitem[Ko13]{Ko13}
S. Kond\={o},
\textit{The Segre cubic and Borcherds products},
Arithmetic and geometry of K3 surfaces and Calabi-Yau threefolds, 549–565,
Fields Inst. Commun., 67, Springer, New York, 2013.

\bibitem[Ko16]{Ko16}
S. Kond\={o},
\textit{The Igusa quartic and Borcherds products},
K3 surfaces and their moduli, 147–170,
Progr. Math., 315, Birkhäuser/Springer,  2016.

\bibitem[KR12]{KR12}
S. Kudla, M. Rapoport,
\textit{On occult period maps},
Pacific J. Math. 260 (2012), no. 2, 565–581.

\bibitem[Li09]{Li09}
L. Li,
\textit{Wonderful compactification of an arrangement of subvarieties},
Michigan Math. J. 58 (2009), no. 2, 535-563.

\bibitem[Lo86]{Lo86}
E. Looijenga,
\textit{New compactifications of locally symmetric varieties}, Proceedings of the 1984 Vancouver conference in algebraic geometry, 341–364,
CMS Conf. Proc., 6, Amer. Math. Soc., Providence, RI, 1986.

\bibitem[Lo03a]{Lo03a}
E. Looijenga,
\textit{Compactifications defined by arrangements. I: The ball quotient case},
Duke Math. J. 118, No. 1, 151-187 (2003).

\bibitem[Lo03b]{Lo03b}
E. Looijenga,
\textit{Compactifications defined by arrangements. II: Locally symmetric varieties of type IV},
Duke Math. J. 119, No. 3, 527-588 (2003).


\bibitem[Ma23]{Ma23}
Y. Maeda,
\textit{Irregular cusps of ball quotients},
Mathematische Nachrichten 296.4 (2023): 1560-1588.


\bibitem[MO23]{MO23}
Y. Maeda, Y. Odaka,
\textit{Fano Shimura varieties with mostly branched cusps},
Springer Proc. Math. Stat. 409, 633-663 (2023).

\bibitem[Ma01]{Ma01}
K. Matsumoto,
\textit{Theta constants associated with the cyclic triple coverings of the complex projective line branching at six points},
Publ. Res. Inst. Math. Sci. 37 (2001), no. 3, 419-440.

\bibitem[MT04]{MT04}
K. Matsumoto, T. Terasoma,
\textit{Theta constants associated to coverings of $\P^1$ branching at eight points},
Compos. Math. 140 (2004), no. 5, 1277–1301.


\bibitem[MS21]{MS21}
H.-B. Moon, L. Schaffler,
\textit{KSBA compactification of the moduli space of K3 surfaces with a purely non-symplectic automorphism of order four},
Proc. Edinb. Math. Soc. (2) 64 (2021), no. 1, 99-127.


\bibitem[Mu03]{Mu03}
S. Mukai,
\textit{An introduction to invariants and moduli},
Cambridge Studies in Advanced Mathematics, 81. Cambridge University Press, Cambridge, 2003.

\bibitem[Mu77]{Mu77}
Mumford, D.
\textit{Hirzebruch's proportionality theorem in the noncompact case}.
Invent. Math. 42 (1977), 239–272.

\bibitem[MFK94]{MFK94}
D. Mumford, J.~Fogarty, F.~Kirwan,
\textit{Geometric invariant theory}, 
Third edition. Ergebnisse der Mathematik und ihrer Grenzgebiete (2), 34. Springer-Verlag, Berlin, 1994.

\bibitem[MM07]{MM07}
A. Mustaţă, M.A. Mustaţă,
\textit{Intermediate moduli spaces of stable maps},
Invent. Math. 167 (2007), no. 1, 47-90.

\bibitem[Na82]{Na82}
I. Naruki,
\textit{Cross ratio variety as a moduli space of cubic surfaces with an appendix by Eduard Looijenga},
Proc. London Math. Soc. (3) 45 (1982), no. 1, 1–30.

\bibitem[Od22]{Od22}
Y. Odaka, \textit{Semi-toric and toroidal compactifications as log minimal models, and applications to weak K-moduli}, arxiv:2203.09120.

\bibitem[Sh53]{Sh53}
G. C. Shephard,
\textit{Unitary groups generated by reflections},
Canad. J. Math. 5 (1953), 364–383.

\bibitem[Th98]{Th98}
W. P. Thurston,
\textit{Shapes of polyhedra and triangulations of the sphere},
The Epstein birthday schrift, 511-549,
Geom. Topol. Monogr., 1, Geom. Topol. Publ., Coventry, 1998.

\bibitem[Zh05]{Zh05}
J. Zhang,
\textit{Geometric compactification of moduli space of cubic surfaces and Kirwan blowup}, 
thesis (Ph.D.)–Rice University. 2005.

\end{thebibliography}
\end{document}